\documentclass[10pt,a4paper]{article}

\usepackage[english]{babel}
\usepackage[T1]{fontenc}

\usepackage{graphicx}
\usepackage{amsmath}
\usepackage{amssymb}
\usepackage{amsthm}
  \theoremstyle{plain}
  \newtheorem{theorem}{Theorem}
  \newtheorem{lemma}{Lemma}
  \newtheorem{proposition}[theorem]{Proposition}

  \newtheorem{example}{Example}[section]
  \theoremstyle{remark}
  \newtheorem*{remark}{Remark}
\usepackage{algorithmic}
\usepackage{algorithm}
\usepackage{booktabs}
\usepackage{subfigure}
\usepackage{mathrsfs}
\usepackage{mathtools}
\usepackage[left=2.7cm,right=2.7cm,top=3.1cm,bottom=3.1cm]{geometry}
\usepackage[colorlinks=true, allcolors=blue, hypertexnames=false]{hyperref}

\usepackage[square,numbers,compress]{natbib}
\usepackage{lipsum}             

\allowdisplaybreaks[4]

\begin{document}

\title{\bf 
Solving Coupled Nonlinear Forward-backward Stochastic Differential Equations:
An Optimization Perspective with Backward Measurability Loss%
\thanks{This work is supported in part by the National Natural Science Foundation of China (62173191).}
}

\author{Yutian Wang%
\thanks{Department of Applied Mathematics, The Hong Kong Polytechnic University,
China. Email: {\tt yutian.wang@connect.\linebreak polyu.hk}}
\and Yuan-Hua Ni%
\thanks{College of Artificial Intelligence, Nankai University,
Tianjin, China. Email: {\tt yhni@nankai.edu.cn}.}
\and Xun Li%
\thanks{Department of Applied Mathematics, The Hong Kong Polytechnic University,
China. Email: {\tt li.xun@polyu.edu.hk}}
}

\date{\today}

\maketitle

\begin{abstract}
This paper aims to extend the BML method proposed in \citet{wang_ni_2022} to make it applicable to more general coupled nonlinear FBSDEs. We interpret BML from the fixed-point iteration perspective and show that optimizing BML is equivalent to minimizing the distance between two consecutive trial solutions in a fixed-point iteration. Thus, this paper provides a theoretical foundation for an optimization-based approach to solving FBSDEs. We also empirically evaluate the method through four numerical experiments. 

\end{abstract}

\section{Introduction}

Forward-backward stochastic differential equations (FBSDEs) are a class of coupled stochastic differential equation systems consisting of forward stochastic differential equations (SDEs) and backward stochastic differential equations (BSDEs). FBSDEs are not only an important theoretical research direction in probability theory but also a powerful tool for solving many practical problems in finance, control, physics, and other fields. In the financial area, we can solve pricing problems for complex financial derivatives using FBSDEs. Unlike the traditional Black-Scholes pricing model, FBSDEs can accommodate realistic situations such as default risk and non-tradable underlying assets, leading to more accurate pricing. The standard maximum principle in stochastic control adopts FBSDEs to characterize the optimal control. In addition, FBSDEs have widespread applications in risk management and financial engineering \cite{Brigo2014, E2019, Beck2023, Germain2021}.

Regarding theoretical research, studying the properties of solutions to FBSDEs is a critical problem, and we are referred to the work \cite{ma_yong_1999} for a comprehensive review of the theory of FBSDEs. Pardoux and Tang provided an approach for solving FBSDEs by establishing a connection between FBSDEs and partial differential equations (PDEs) \cite{pardoux_tang_1999}; see also the ``four-step scheme'' approach studied in \cite{ma_protter_1994}. Based on this connection, solving FBSDEs can be transformed into solving PDEs, which can then be solved using numerical methods such as finite element, finite difference, or sparse grid methods. In terms of numerical computation, the solution to FBSDEs is often presented as conditional expectations, thus requires the numerical computation of nested conditional expectations; some commonly used methods are mentioned in the literature, including Monte Carlo methods, quantization tree methods, methods based on Malliavin calculus or kernel estimation regression, function space projection methods, integration methods in Wiener space, and Wiener chaos expansion methods. Among these, the function space projection method and the Wiener chaos expansion method have advantages in terms of computational complexity, but so far no single method has been able to simultaneously address the issues of dimensionality and accuracy simultaneously \cite{Bally2003, Bouchard2004, Gobet2005, Crisan2012, Geiss2016}.

In recent years, with the popularity of machine learning and deep learning, some researchers have discovered that neural networks can overcome the difficulties encountered by traditional FBSDE computational methods when dealing with high-dimensional problems. It is because conventional methods generally rely on grid discretization of the spatial domain, which leads to an exponential increase in the number of nodes to be considered for each computation as the problem dimension increases, making it unsuitable for high-dimensional problems. On the other hand, the recently proposed deep BSDE method in \citet{Han2018} does not require spatial discretization, which is based on neural network optimization and gives it an advantage over traditional methods when dealing with high-dimensional problems. In the deep BSDE method, the BSDE is treated as a forward SDE, and the values of the random process to be solved at each time step are learned by minimizing a global loss function. This method can handle general BSDEs and can also be used to solve high-dimensional PDE problems, such as the neural network proposed in \citet{Chan2019} for solving fixed-point problems. In addition, the BSDE solution method based on the multilayer Picard iteration has also received extensive  attention \cite{E2019} and it can more effectively solve high-dimensional, nonlinear PDE problems. Interestingly, these algorithms have good convergence and accuracy in high-dimensional cases.

This paper focuses on developing the Deep BSDE method to make it applicable to coupled FBSDEs. Before summarizing the contributions of this paper, several references related to the discretization of FBSDEs and the Deep BSDE method are provided. \citet{Han2018} proposed the Deep BSDE method and applied it to compute solutions for three decoupled nonlinear FBSDEs, including the Black-Scholes equation in finance, the Allen-Cahn equation in physics, and the HJB equation in optimal control. \citet{han_long_2020} established a theoretically analytical framework for the Deep BSDE method and provided the posteriori error estimates. \citet{bender_zhang_2008} proposed a Picard iteration scheme for coupled FBSDEs and provided the convergence properties characterized by the time discretization step $n$ and the number $m$ of iterations. \citet{ji_peng_2020_IIS} designed three algorithms for coupled nonlinear FBSDEs based on the Deep BSDE method. These algorithms involve different parameterizations of the equation's solution, including considering only forward state feedback, both forward and backward state feedback, and considering feedback from the forward, backward, and control processes together. The recent work \cite{Andersson2023} suggested that the vanilla Deep BSDE method fails at coupled FBSDE cases and proposes a reformulated loss, which is a weighted sum of the initial value and a variance term.  \citet{wang_ni_2022} built a probabilistic framework for Howard's policy iteration algorithm using the language of FBSDEs,
where the Backward Measurability Loss (BML)  is proposed for solving decoupled linear FBSDEs.

The main contribution of this paper is the introduction of BML for coupled nonlinear FBSDEs and an analysis of its theoretical properties. Similar to the optimization-based approach used to solve FBSDEs, the proposed BML can be interpreted as the distance between the current estimated solution and the next solution in fixed-point iteration (Picard iteration) procedure. Particularly, when the BML value is zero, the current estimated solution can be understood as the fixed point of Picard iteration, i.e., the solution of the equation under consideration. The difference with the Deep BSDE method lies in the fact that the proposed BML measures the satisfaction of FBSDE by both the backward process and the control process, while the Deep BSDE method only considers the loss function for the initial value of the backward process and the control process. By optimizing the BML value, this paper attempts to minimize the comprehensive error between the estimated backward and control processes, and the true solution process over the entire time interval. On the other hand, the Deep BSDE method only focuses on minimizing the distance between the estimated value of the backward process and the true value at the initial time point. In other words, this paper focuses more on obtaining the complete solution of FBSDE over the entire time interval, while the Deep BSDE method is more concerned with the value of the backward process of FBSDE at the initial time point. Due to the difference in focus, the estimated accuracy of the solution obtained by optimizing the BML value may not be as good as directly applying the Deep BSDE method at the initial time point. However, the research in this paper provides a feasible approach for problems that require the entire trajectory of the solution of FBSDEs.

The organization of this paper is as follows. Section 2 describes the problem of solving coupled FBSDEs discussed. Section 3 provides a detailed introduction to the proposed BML and analyzes its properties for FBSDEs ranging from simple to general forms. Section 4 presents the specific computational format and algorithm flow based on optimizing the BML. Section 5 applies the proposed method to compute several specific FBSDEs. Finally, Section 6 summarizes the content of this paper.

Some commonly used notations in this paper are summarized as follows:

\begin{itemize}
  \item Let $\mathbb{F}=\{\mathcal{F}_t\}_{0\leq t\leq T}$ denote the natural
      filtration generated by the standard Brownian motion.
  \item Let $L^2_{\mathcal{F}}(\Omega;N)$ be the set of random variables
      $f:\Omega\to N$ satisfying the following conditions: (i) $f$ is
      measurable with respect to the $\sigma$-algebra $\mathcal{F}$, and (ii) $f$ is
      square integrable, i.e., $\operatorname{\mathbb{E}}|f|^2 < \infty$.
  \item Let $L^2_{\mathcal{F}}(0,T;N)$ be the set of stochastic processes
      $X:\Omega\times[0,T]\to N$ satisfying the following conditions: (i) $X$ is
      adapted to the filtration $\mathbb{F}$, and (ii) $\int_0^T
      \operatorname{\mathbb{E}}|X(t)|^2\,dt < \infty$.
  \item Let $L^2_{\mathcal{F}}(\Omega; C[0,T];N)$ be the set of stochastic
      processes $X:\Omega\times[0,T]\to N$ satisfying the following conditions: (i)
      $X$ is adapted and continuous with respect to the filtration
      $\mathbb{F}$, and (ii) $\operatorname{\mathbb{E}}\sup_{t\in[0,T]}|X(t)|^2 <
      \infty$.
  \item Let $L^2_{\mathcal{F}_T}(\Omega;W^{1,\infty}(M;N))$ be the set of
      functions $g:\Omega\times M\to N$ satisfying the following conditions: (i)
      $g(\theta)$ is uniformly continuous with respect to $\theta$, i.e., there
      exists a constant $L_g$ such that for any $\theta_1,\theta_2\in M$, the
      inequality $|g(\theta_1) - g(\theta_2)| \leq L_g|\theta_1 - \theta_2|$ holds
      almost surely, (ii) $g$ is $\mathcal{F}_T$-measurable for any fixed $\theta$,
      and (iii) if $\theta=0$ is fixed, $g\in L^2_{\mathcal{F}_T}(\Omega;N)$.
  \item Let $L^2_{\mathcal{F}}(0,T;W^{1,\infty}(M;N))$ be the set of
      functions $f:\Omega\times[0,T]\times M\to N$ satisfying the following
      conditions: (i) $f(t,\theta)$ is uniformly continuous with respect to $\theta$, i.e.,
      there exists a constant $L_f$ such that for any $\theta_1,\theta_2\in M$, the
      inequality $|f(t,\theta_1) - f(t,\theta_2)| \leq L_f|\theta_1 - \theta_2|$
      holds, (ii) $f$ is adapted to the filtration $\mathbb{F}$ for any fixed
      $\theta$, and (iii) if $\theta=0$ is fixed, $f\in L^2_{\mathcal{F}}(0,T;N)$.
  \item Let $\mathcal{M}[0,T]$ be the product space $L^2_{\mathcal{F}}(\Omega;
      C[0,T];\mathbb{R}^m)\times L^2_{\mathcal{F}}(0,T;\mathbb{R}^{m\times d})$.
\end{itemize}

\section{Problem Formulation}

This paper considers the following FBSDE
\begin{equation}
\label{eq:fully-coupled-nonlinear-FBSDE}
\left\{
\begin{aligned}
X_t &= x_0 + \int_0^tb(s,X_s,Y_s,Z_s)\,ds +
\int_0^t\sigma(s,X_s,Y_s,Z_s)\,dW_s,\\ Y_t &= g(X_T) +
\int_t^Tf(s,X_s,Y_s,Z_s)\,ds - \int_t^TZ_s\,dW_s.
\end{aligned}
\right.
\end{equation}
Here, $X$ is a stochastic process taking values in a $n$-dimensional Euclidean space, commonly referred to as the \textit{forward process} or \textit{state}; $Y$ is a stochastic process taking values in a $m$-dimensional Euclidean space, commonly referred to as the \textit{backward process}; and $Z$ is a $\mathbb{R}^{m\times d}$-valued stochastic process\footnote{In some literature, this process is referred to as the control process of FBSDE. However, in order to avoid confusion with the control process $\{u_t,t\geq 0\}$ in stochastic optimal control problems, this terminology is not adopted in this paper.}. The functions $b,\sigma,f$, and $g$ are measurable functions appropriate in dimension. Here, $b$ and $\sigma$ are referred to as the drift and diffusion coefficient, respectively, $f$ is the generator of the backward equation, and $g$ is the terminal condition. The $x_0\in\mathbb{R}^n$ is the initial condition of the forward equation.

Strictly speaking, the FBSDE in
Eq.~\eqref{eq:fully-coupled-nonlinear-FBSDE} is defined on a complete probability space 
$(\Omega,\mathcal{F},\mathbb{P},\mathbb{F})$, which admits a $d$-dimensional
standard Brownian motion $W$ and its associated natural filtration
$\mathbb{F}=\{\mathcal{F}_t\}_{0\leq t\leq T}$.
The finite positive number $T\in(0,\infty)$
is referred to as the terminal time.
The triplet $(X,Y,Z)\in L^2_{\mathcal{F}}(\Omega;
  C[0,T];\mathbb{R}^n) \times L^2_{\mathcal{F}}(\Omega;
  C[0,T];\mathbb{R}^m) \times
  L^2_{\mathcal{F}}(0,T;\mathbb{R}^{m\times d})$ is called a solution to
FBSDE~\eqref{eq:fully-coupled-nonlinear-FBSDE}, if and only if, for any time $t\in[0,T]$, 
FBSDE~\eqref{eq:fully-coupled-nonlinear-FBSDE} holds almost surely.
The following theorem provides the existence and uniqueness of solutions to coupled FBSDEs in a sufficiently small time interval.

\begin{theorem}[\cite{ma_yong_1999}]
  \label{th:uniqueness-and-existence-fully-coupled-nonlinear-FBSDE}
  Let the terminal condition $g\in
  L^2_{\mathcal{F}}(\Omega;W^{1,\infty}(\mathbb{R}^n;\mathbb{R}^m)$,
  and suppose that
\begin{equation*}
\left\{
\begin{aligned}
b &\in
L^2_{\mathcal{F}}(0,T;W^{1,\infty}(\mathbb{R}^n\times\mathbb{R}^m\times\mathbb{R}^{m\times
  d};\mathbb{R}^n),\\ \sigma &\in
L^2_{\mathcal{F}}(0,T;W^{1,\infty}(\mathbb{R}^n\times\mathbb{R}^m\times\mathbb{R}^{m\times
  d};\mathbb{R}^m),\\ f &\in
L^2_{\mathcal{F}}(0,T;W^{1,\infty}(\mathbb{R}^n\times\mathbb{R}^m\times\mathbb{R}^{m\times
  d};\mathbb{R}^{m\times d}).
\end{aligned}
\right.
\end{equation*}
Moreover, assume that there exist constants $L_0 $ and $L_g$  such that
the following inequalities hold almost surely:
\begin{equation*}
\left\{
  \begin{aligned}
  &|\sigma(t,x,y,\hat{z}) - \sigma(t,x,y,\check{z})| \leq
    L_0|\hat{z} - \check{z}|,\quad t\ \textrm{a.e. on}\ (0,\infty),~ \forall
    (x,y)\in\mathbb{R}^n\times\mathbb{R}^m,~\hat{z},\check{z}\in\mathbb{R}^{m\times d},\\
  &|g(\hat{x}) - g(\check{x})| \leq
  L_g|\hat{x} - \check{x}|,\quad\quad\quad\quad\quad\quad~\forall \hat{x},\check{x}\in\mathbb{R}^n.
  \end{aligned}
\right.
\end{equation*}
If $L_0L_g < 1$, then there exists a constant $T_0 > 0$ such that for any
$T\in(0,T_0]$ and any initial condition
  $x_0\in\mathbb{R}^n$, FBSDE~\eqref{eq:fully-coupled-nonlinear-FBSDE}
  admits a unique solution.
\end{theorem}

\begin{remark}
It should be noted that this theorem only provides the existence and uniqueness of solutions for small enough time intervals, i.e., $T$ must be smaller than a certain constant $T_0$ that depends on the problem itself. The proof is based on constructing a contraction mapping and applying Picard's iteration. Engaging readers may refer to the monograph \cite{ma_yong_1999} for technical details.

Besides this contraction mapping approach, two methods exist for studying the existence and uniqueness of FBSDEs, the ``four-step scheme'' approach and the ``monotone assumptions'' approach. The former verifies the existence and uniqueness by expressing the solution with a corresponding PDE's solution \cite{ma_protter_1994}. This method may deal with arbitrary time duration but is not applicable to FBSDEs with random coefficients. The latter, however, tackles the existence and uniqueness using a completely different approach by assuming a monotone condition \cite{peng_wu_1999,pardoux_tang_1999}. Our paper does not intend to provide new results in this direction. Instead, the rigorous foundation of our work may build upon these methods. Indeed, all the conditions required to obtain our results can be easily verified under assumptions made by any of the three methods which ensures the existence and uniqueness of the solution.
\end{remark}

The subsequent parts of this paper aim to study the problem of finding the solution $(X,Y,Z)$ to FBSDE~\eqref{eq:fully-coupled-nonlinear-FBSDE} with a given time interval $[0,T]$, initial condition $x_0$, terminal condition $g$, drift coefficient $b$, diffusion coefficient $\sigma$, and generator $f$. In particular, the focus is on designing numerical algorithms that can sample the stochastic process $(X,Y,Z)$. More specifically, the goal is to devise specific computational processes that can be used to obtain sample paths of the solution.

\section{Backward Measurability Loss}

Coupled FBSDEs cannot be directly solved by discretizing the time, as commonly done with first-order differential equations $\dot{x}=f(t,x)$. If we directly discretize the coupled FBSDEs, the "increment" on the right-hand side of the equation would be related to the current values of $X$, $Y$, and $Z$, while only $x_0$ is known as the initial condition, and the initial conditions of $Y$ and $Z$ are unknown. Therefore, we cannot start from the initial time and iteratively compute the differential relationship the equation provides for time $t$.

In this paper, we consider solving the coupled
FBSDE~\eqref{eq:fully-coupled-nonlinear-FBSDE} using an optimization-based method,
and the intuition is given below.
Let $(X,Y,Z)$ denote the true solution to the equation. Given a pair of
stochastic processes $(\tilde{y},\tilde{z}) \in L^2_{\mathcal{F}}(\Omega;
C[0,T];\mathbb{R}^m) \times L^2_{\mathcal{F}}(0,T;\mathbb{R}^{m\times d})$ as a
trial solution, we are interested in measuring the discrepancy between $Y$
and $\tilde{y}$, as well as $Z$ and $\tilde{z}$. Specifically, we seek to find a
distance function $\operatorname{dist}((\tilde{y},\tilde{z}), (Y,Z))$ defined on
the space $L^2_{\mathcal{F}}(\Omega; C[0,T];\mathbb{R}^m) \times
L^2_{\mathcal{F}}(0,T;\mathbb{R}^{m\times d})$ (denoted by $\mathcal{M}[0,T]$
for simplicity) to represent this discrepancy.
Next, by solving the optimization problem \begin{equation} \label{op:dist-yz-YZ}
\min_{\tilde{y},\tilde{z}} \quad \operatorname{dist}((\tilde{y},\tilde{z}),
(Y,Z)), \end{equation} we can obtain the solution $(Y,Z)$. Finally, substituting
$(Y,Z)$ into the forward equation of the coupled FBSDE, we obtain the forward process
$X$.

In summary, from the optimization perspective, we need to address two issues: 
finding an appropriate metric $\operatorname{dist}(\cdot,\cdot)$
and solving the optimization problem \eqref{op:dist-yz-YZ}.
The theoretical part of this paper mainly discusses how to implement the first
step, while for the optimization problem in the second step, we choose to apply
neural network optimization methods. In fact, the difficulty
of the optimization perspective lies in finding a
suitable \textit{computationally feasible} metric. By \textit{computationally
feasible}, we mean that, given a trial solution $(\tilde{y},\tilde{z})$, we can
compute the discrepancy $\operatorname{dist}((\tilde{y},\tilde{z}), (Y,Z))$ even
if we do not know the true solution $(Y,Z)$. This requirement of
computational feasibility is evident because we do not know the true solution
$(Y,Z)$ of the coupled FBSDE.  Therefore, as a compromise, we seek an
alternative to $\operatorname{dist}((\tilde{y},\tilde{z}), (Y,Z))$, an error
function $\ell$ that only depends on the trial solution and simultaneously
consider the optimization problem 
\begin{equation}
\label{op:ell-yz}
\min_{\tilde{y},\tilde{z}} \quad \ell(\tilde{y},\tilde{z}).
\end{equation}

The contribution of this paper lies in finding a specific error function
$\ell$ called the Backward Measurability Loss (BML), which is defined by
\begin{equation}
\label{eqdef:fully-coupled-nonlinear-BML}
\operatorname{BML}(\tilde{y},\tilde{z};\mu) \coloneqq \int_0^T
\operatorname{\mathbb{E}} \left|\tilde{y}_t - \left(g(\widetilde{X}_T) +
\int_t^T {f}(s,\widetilde{X}_s,\tilde{y}_s,\tilde{z}_s)\,ds - \int_t^T
\tilde{z}_s\,dW_s\right) \right|^2 \mu(dt),
\end{equation}
with a $\sigma$-finite measure $\mu$ on $[0,T]$ and $\widetilde{X}$ the solution to the SDE
\begin{equation}
\label{eq:BML-tilde-X}
\widetilde{X}_t = x_0 + \int_0^t
b(s,\widetilde{X}_s,\tilde{y}_s,\tilde{z}_s)\,ds + \int_0^t
\sigma(s,\widetilde{X}_s,\tilde{y}_s,\tilde{z}_s)\,dW_s.
\end{equation}
Next, we analyze the physical interpretation of the BML by considering three
cases: decoupled linear FBSDE, decoupled nonlinear FBSDE, and fully coupled
nonlinear FBSDE. Specifically, for the decoupled linear FBSDE, we construct a
metric $\operatorname{dist}\coloneqq\operatorname{dist}_\mu$ ($\mu$ being any
$\sigma$-finite measure on $[0,T]$)\footnote{Strictly speaking,
$\sqrt{\operatorname{dist}_\mu}$ satisfies the properties of a pseudometric.
However, for computational
convenience, we choose to optimize the squared version of it.}:
\begin{equation}
\label{eqdef:dist-yz-YZ}
\operatorname{dist}_\mu((\hat{y},\hat{z}), (\check{y},\check{z})) = \int_0^T
\operatorname{\mathbb{E}} \left\{ |\hat{y}_t - \check{y}_t|^2 + \int_t^T
|\hat{z}_s - \check{z}_s|^2 \,ds \right\}\, \mu(dt).
\end{equation}
It is then shown that for any $(\tilde{y},\tilde{z}) \in \mathcal{M}[0,T]$, we have
\begin{equation}
\label{eq:strict-eq-dist-BML}
\operatorname{dist}_\mu((\tilde{y},\tilde{z}), (Y,Z)) = \operatorname{BML}(\tilde{y},\tilde{z};\mu),
\end{equation}
as stated in Theorem \ref{th:BML-interpretation-simple-decoupled-FBSDE}. In this
case, the optimization problems \eqref{op:dist-yz-YZ} and \eqref{op:ell-yz} are
completely equivalent.

For nonlinear FBSDEs, including fully coupled nonlinear FBSDE, although we
may not find a perfect explanation for the BML as that in Equation
\eqref{eq:strict-eq-dist-BML}, we can still interpret the BML value and explain
the physical meaning of optimizing the BML value from the perspective of
fixed-point iterations. Specifically, we construct a mapping $\Phi$ defined on
$\mathcal{M}[0,T]$ such that for any pair of stochastic processes
$(\tilde{y},\tilde{z})$, it is a fixed point of $\Phi$ if and only if it is a
solution to the backward equation. It is then proven that for any
$(\tilde{y},\tilde{z}) \in \mathcal{M}[0,T]$, we have \begin{equation}
\label{eq:indirect-eq-dist-BML} \operatorname{dist}_\mu((\tilde{y},\tilde{z}),
\Phi(\tilde{y},\tilde{z})) = \operatorname{BML}(\tilde{y},\tilde{z};\mu),
\end{equation} as stated in Proposition
\ref{prop:BML-interpretation-nonlinear-decoupled-FBSDE} and Proposition
\ref{prop:BML-interpretation-fully-coupled-nonlinear-FBSDE}.

\subsection{Simple Decoupled FBSDE}
\label{ssec:simple-decoupled-FBSDE}

Let's first consider a special case of the original equation
\eqref{eq:fully-coupled-nonlinear-FBSDE}: the drift coefficient $b$ and
diffusion coefficient $\sigma$ in the forward equation do not depend on the
solution of the backward equation, and the generator $f$ and terminal condition
$g$ in the backward equation are known stochastic processes $\{f_t,0\leq t\leq
T\}$ and random variable $\xi$, respectively, independent of the solution
$(X,Y,Z)$. In this case, the desired equations can be written
as
\begin{equation}
\label{eq:simple-decoupled-FBSDE}
\left\{\begin{aligned}
X_t &= x_0 + \int_0^tb(s,X_s)\,ds + \int_0^t\sigma(s,X_s)\,dW_s,\\ Y_t
&= \xi + \int_t^Tf_s\,ds - \int_t^TZ_s\,dW_s.
\end{aligned} \right.
\end{equation}
This represents a decoupled FBSDE, where
the forward and backward equations can be solved separately. When computing the
BML (Backward Memorization Loss) value, the $\widetilde{X}$ in equation
\eqref{eq:BML-tilde-X} represents the forward process $X$ of this FBSDE.
Therefore, we assume that $X$ is known and focus on discussing the solution
$(Y,Z)$ of the backward equation.

Considering the decoupled FBSDE \eqref{eq:simple-decoupled-FBSDE}, according to
definition \eqref{eqdef:fully-coupled-nonlinear-BML}, for any trial solution
$(\tilde{y},\tilde{z})\in \mathcal{M}[0,T]$, its BML value is defined as $$
\operatorname{BML}(\tilde{y},\tilde{z};\mu) = \int_0^T \operatorname{\mathbb{E}}
\left|\tilde{y}_t - \biggl(\xi + \int_t^T{f}_s\,ds -
\int_t^T\tilde{z}_s\,dW_s\biggr) \right|^2\mu(dt).$$

Let's denote the solution of the BSDE as $(Y,Z)$, and define
a metric $\operatorname{dist}_\mu$ on the space $\mathcal{M}[0,T]$ as given in
equation \eqref{eqdef:dist-yz-YZ}. Then the following theorem relates the BML
value to this metric.

\begin{theorem}
    \label{th:BML-interpretation-simple-decoupled-FBSDE} 
Let $\mu$ be a $\sigma$-finite measure on $[0,T]$, and consider the backward equation in FBSDE \eqref{eq:simple-decoupled-FBSDE}.
If $\xi\in L^2_{\mathcal{F}_T}$ and
$f\in L^2_{\mathcal{F}}(0,T;\mathbb{R}^m)$, then for any trial solution
$(\tilde{y},\tilde{z})\in \mathcal{M}[0,T]$, its BML value is equal to the
distance between the trial solution and the true solution $(Y,Z)$, i.e.,
equation \eqref{eq:strict-eq-dist-BML} holds.
\end{theorem}

\begin{proof}
    Substituting the backward equation into the expression of the BML
value, we have
\begin{align*} \operatorname{BML}(\tilde{y},\tilde{z};\mu) &=
\int_0^T \operatorname{\mathbb{E}} \left|\tilde{y}_t - \biggl(\xi +
\int_t^T{f}_s\,ds - \int_t^T\tilde{z}_s\,dW_s\biggr) \right|^2\mu(dt)\\ &=
\int_0^T \operatorname{\mathbb{E}} \left|\tilde{y}_t - Y_t + \biggl(\xi +
\int_t^T{f}_s\,ds - \int_t^TZ_s\,dW_s\biggr) \right.\\ &\left.
\hphantom{\,=\int_0^T \operatorname{\mathbb{E}} \biggl|} -
\biggl(\xi + \int_t^T{f}_s\,ds - \int_t^T\tilde{z}_s\,dW_s\biggr)
\right|^2\mu(dt) \\ &= \int_0^T \operatorname{\mathbb{E}} \left|\tilde{y}_t -
Y_t + \biggl(\int_t^TZ_s\,dW_s - \int_t^T\tilde{z}_s\,dW_s\biggr)
\right|^2\mu(dt)\\ &= \int_0^T \operatorname{\mathbb{E}} \left\{|\tilde{y}_t -
Y_t|^2 + \biggl|\int_t^TZ_s\,dW_s - \int_t^T\tilde{z}_s\,dW_s\biggr|^2
\right\}\mu(dt)\\ &= \int_0^T \operatorname{\mathbb{E}} \left\{|\tilde{y}_t -
Y_t|^2 + \int_t^T|Z_s-\tilde{z}_s|^2\,ds\right\}\mu(dt).
\end{align*}
The third equality holds because the cross terms when expanding the squares cancel out.
\end{proof}

\begin{remark} 
    This theorem demonstrates that for FBSDEs of form
\eqref{eq:simple-decoupled-FBSDE}, if we take the BML value of a trial solution
$(\tilde{y},\tilde{z})$ as the error function $\ell(\cdot,\cdot)$, then the
optimization problems \eqref{op:dist-yz-YZ} and \eqref{op:ell-yz} are completely
equivalent, as their objective functions are strictly equal. In other words,
given two sets of trial solutions $(\tilde{y}^{(1)},\tilde{z}^{(1)})$ and
$(\tilde{y}^{(2)},\tilde{z}^{(2)})$, if we compute their BML values as $\ell_1$
and $\ell_2$, respectively, and $\ell_1 < \ell_2$, we can say that
$(\tilde{y}^{(1)},\tilde{z}^{(1)})$ is closer to the true solution $(Y,Z)$ than
$(\tilde{y}^{(2)},\tilde{z}^{(2)})$, even without knowing the specific values of
$(Y,Z)$.
\end{remark}

\begin{remark}
    The choice of the time measure $\mu$ can be determined based on
specific requirements. For example, if simplicity is preferred, $\mu$ can be
chosen as the Dirac measure at the starting time, $\delta(dt)$, which eliminates
the time integration in the BML definition, i.e.,
$$\operatorname{BML}(\tilde{y},\tilde{z};\delta) =
\operatorname{\mathbb{E}} \left|\tilde{y}_0 - \biggl(g(\widetilde{X}_T) +
\int_t^T{f}(s, \widetilde{X}_s,\tilde{y}_s,\tilde{z}_s)\,ds -
\int_t^T\tilde{z}_s\,dW_s\biggr) \right|^2.$$
If a balanced evaluation of
the error over the entire time interval is desired, $\mu$ can be chosen as the
Lebesgue measure $dt$. If both a comprehensive evaluation of the error over the
entire time interval and an emphasis on errors near the origin are desired,
$\mu$ can be chosen as an exponentially decaying measure $e^{-\gamma t}\,dt$.
For more details, please refer to \ref{ssec:calc-BML} on the computation of BML
values.
\end{remark}


\subsection{General Decoupled FBSDE}

In this section, we consider another special case of the original equation
\eqref{eq:fully-coupled-nonlinear-FBSDE}: the decoupled nonlinear FBSDE. It is
the nonlinear version of the case considered in the previous section, relaxing
the restriction on generator $f$, which is now allowed to depend on the
solution $(Y,Z)$ of the backward equation. The equation to be solved can be
written as follows
\begin{equation}
\label{eq:nonlinear-decoupled-FBSDE}
\left\{
\begin{aligned}
X_t &= x_0 + \int_0^tb(s,X_s)\,ds + \int_0^t\sigma(s,X_s)\,dW_s,\\
Y_t &= \xi + \int_t^Tf(s,Y_s,Z_s)\,ds - \int_t^TZ_s\,dW_s.
\end{aligned}
\right.
\end{equation}
Note that in this section, we still do not allow $f$ to depend on the solution
$X$ of the forward equation, so we are considering a decoupled FBSDE. Following
the approach used in the previous section, we assume that $X$ is known and focus
on discussing the solution $(Y,Z)$ of the backward equation.

Consider the decoupled nonlinear FBSDE \eqref{eq:nonlinear-decoupled-FBSDE}.
According to the definition \eqref{eqdef:fully-coupled-nonlinear-BML}, for any
trial solution $(\tilde{y},\tilde{z})\in \mathcal{M}[0,T]$, its BML value is
given by
$$ \operatorname{BML}(\tilde{y},\tilde{z};\mu) = \int_0^T
\operatorname{\mathbb{E}} \left|\tilde{y}_t - \biggl(\xi +
\int_t^Tf(s,\tilde{y}_s,\tilde{z}_s)\,ds -
\int_t^T\tilde{z}_s\,dW_s\biggr) \right|^2\mu(dt).$$

Let's denote the solution of the BSDE as $(Y,Z)$, and define
the metric $\operatorname{dist}_\mu$ on the space $\mathcal{M}[0,T]$ as in
equation \eqref{eqdef:dist-yz-YZ}. Unfortunately, for the nonlinear case, we
cannot directly equate the BML value with $\operatorname{dist}_\mu$. In fact,
the connection between these two quantities is given by the following
proposition.

\begin{proposition}
\label{prop:BML-interpretation-nonlinear-decoupled-FBSDE}
Let $\mu$ be a $\sigma$-finite measure on $[0,T]$.
Consider the backward equation in FBSDE
\eqref{eq:nonlinear-decoupled-FBSDE}.
Assume that the terminal condition $\xi\in L^2_{\mathcal{F}_T}$ and the
generator $f\in
L^2_{\mathcal{F}}(0,T;W^{1,\infty}(\mathbb{R}^m\times\mathbb{R}^{m\times
d};\mathbb{R}^m))$. Then, for any trial solution $(\tilde{y},\tilde{z})\in
\mathcal{M}[0,T]$, its BML value is equal to its distance to the true solution
$(\widetilde{Y},\widetilde{Z})$ of a linear BSDE, i.e.,
$$\operatorname{BML}(\tilde{y},\tilde{z};\mu) = 
\operatorname{dist}_\mu((\tilde{y},\tilde{z}), (\widetilde{Y},\widetilde{Z})).$$
Here, $\operatorname{dist}_\mu$ is defined by equation \eqref{eqdef:dist-yz-YZ},
and $(\widetilde{Y},\widetilde{Z})$ satisfies the linear BSDE:
\begin{equation}
\label{eq:BML-linearized-BSDE}
\widetilde{Y}_t = \xi + \int_t^Tf(s,\tilde{y}_s,\tilde{z}_s)\,ds - \int_t^T\widetilde{Z}_s\,dW_s.
\end{equation}
\end{proposition}

The key to proving this proposition is to recognize that the BML value of a trial solution is relative to the FBSDE considered. Even for the same trial solution $(\tilde{y},\tilde{z})$, its definition of BML value may differ depending on the FBSDE under consideration. However, if two FBSDEs under consideration are related to each other through $(\tilde{y},\tilde{z})$, then the trial solution $(\tilde{y},\tilde{z})$ can have equal BML values for both FBSDEs.

\begin{proof}
Fix stochastic processes $\tilde{y}$ and $\tilde{z}$. Let
$\tilde{f}_s\coloneqq f(s,\tilde{y}_s,\tilde{z}_s)$, and consider the BSDE
\begin{equation}
\label{eq:proof-tmp-311}
\widehat{Y}_t = \xi + \int_t^T\tilde{f}_s\,ds - \int_t^T\widehat{Z}_s\,dW_s.
\end{equation}
We want to apply Theorem~\ref{th:BML-interpretation-simple-decoupled-FBSDE} to
equation \eqref{eq:proof-tmp-311}, which requires verifying  $\tilde{f}_s\in
L^2_{\mathcal{F}}(0,T;\mathbb{R}^n)$. Based on the assumption $f\in
L^2_{\mathcal{F}}(0,T;W^{1,\infty}(\mathbb{R}^m\times\mathbb{R}^{m\times
d};\mathbb{R}^m))$, and since $\tilde{y}$ and $\tilde{z}$ are both
$\mathbb{F}$-adapted processes, we conclude that $\tilde{f}_s$ is also
$\mathbb{F}$-adapted. Moreover,
$$\operatorname{\mathbb{E}}\int_0^T|\tilde{f}_s -f(s,0,0)|^2\,ds \leq L_f^2
\operatorname{\mathbb{E}}\int_0^T\bigl||\tilde{y}_s| + |\tilde{z}_s|\bigr|^2\,ds
< \infty.$$
Hence, we have $\tilde{f}_s-f(s,0,0)\in L^2_{\mathcal{F}}(0,T;\mathbb{R}^n)$,
which implies  $\tilde{f}_s\in L^2_{\mathcal{F}}(0,T;\mathbb{R}^n)$.
Therefore, according to
Theorem~\ref{th:BML-interpretation-simple-decoupled-FBSDE}, we have
$$\int_0^T \operatorname{\mathbb{E}} \left|\tilde{y}_t - \biggl(\xi +
\int_t^T\tilde{f}_s\,ds - \int_t^T\tilde{z}_s\,dW_s\biggr) \right|^2\mu(dt) =
\operatorname{dist}_\mu((\tilde{y},\tilde{z}),(\widehat{Y},\widehat{Z})).$$
Since $\tilde{f}_s=f(s,\tilde{y}_s,\tilde{z}_s)$, the left-hand side of the
equation is also equal to the BML value of the trial solution with respect to
FBSDE \eqref{eq:nonlinear-decoupled-FBSDE}. Finally, since $\xi\in
L^2_{\mathcal{F}_T}$ and $\tilde{f}_s\in L^2_{\mathcal{F}}(0,T;\mathbb{R}^n)$,
linear BSDE \eqref{eq:proof-tmp-311} has a unique solution, i.e.,
$\widehat{Y}=\widetilde{Y}$ and $\widehat{Z}=\widetilde{Z}$. Therefore, the
right-hand side of the equation is equal to
$\operatorname{dist}_\mu(\widetilde{Y},\widetilde{Z})$. This concludes the proof
of the proposition.
\end{proof}

\begin{remark}
In comparison to Theorem~\ref{th:BML-interpretation-simple-decoupled-FBSDE}, the conclusion of this proposition includes $(\widetilde{Y},\widetilde{Z})$ in the position of $(Y,Z)$ on the right-hand side of equation \eqref{eq:strict-eq-dist-BML}. It is expected because when computing the BML value of a trial solution to nonlinear FBSDE \eqref{eq:nonlinear-decoupled-FBSDE}, we substitute $(\tilde{y},\tilde{z})$ for the unknown $(Y,Z)$. It leads to the BML value of the trial solution being the same as its BML value to linear BSDE \eqref{eq:BML-linearized-BSDE}. Although the generators of these two BSDEs are different, when calculating the BML value, the integral term with respect to time is the same for both the simple BSDE \eqref{eq:proof-tmp-311} and the nonlinear BSDE \eqref{eq:nonlinear-decoupled-FBSDE}, given by
$$
\int_t^T\tilde{f}_s\,ds = \int_t^Tf(s,\tilde{y}_s,\tilde{z}_s)\,ds,
$$ 
while the other terms are completely identical. Based on the conclusion from the previous section, in this case, the BML value is necessarily equal to the distance between the trial solution and the solution of BSDE \eqref{eq:BML-linearized-BSDE}.

It is important to note that even though the BML value in this case equals the distance between the trial solution $(\tilde{y},\tilde{z})$ and another pair of stochastic processes $(\widetilde{Y},\widetilde{Z})$, optimizing this BML value does not necessarily lead to the convergence to $(\widetilde{Y},\widetilde{Z})$ and ultimately deviate from the true solution $(Y,Z)$. Since $(\widetilde{Y},\widetilde{Z})$ depends on the current trial solution $(\tilde{y},\tilde{z})$, any changes in the trial solution will also change $(\widetilde{Y},\widetilde{Z})$, so there is no fixed value for it. In fact, if $(\tilde{y},\tilde{z})$ does converge to an optimal solution $(\tilde{y}^*,\tilde{z}^*)$, the corresponding $(\widetilde{Y},\widetilde{Z})$ will converge to the desired true solution $(Y,Z)$, as stated in Theorem~\ref{th:BML-interpretation-nonlinear-decoupled-FBSDE}.
\end{remark}

Now, we explain the BML value in the nonlinear case from the perspective of fixed-point iteration. When the assumptions of Proposition \ref{prop:BML-interpretation-nonlinear-decoupled-FBSDE} hold, for any pair of stochastic processes $(\tilde{y},\tilde{z})\in\mathcal{M}[0,T]$, the linear BSDE \eqref{eq:BML-linearized-BSDE} has a unique solution $(\widetilde{Y},\widetilde{Z})$. Thus, we can define a mapping $\Phi$ on the space $\mathcal{M}[0,T]$ that maps $(\tilde{y},\tilde{z})$ to the unique solution $(\widetilde{Y},\widetilde{Z})$ of linear BSDE \eqref{eq:BML-linearized-BSDE}. According to this definition, the solution $(Y,Z)$ of FBSDE \eqref{eq:nonlinear-decoupled-FBSDE} is necessarily a fixed point of mapping $\Phi$. Conversely, any fixed point of mapping $\Phi$ must satisfy the backward equations in FBSDE \eqref{eq:nonlinear-decoupled-FBSDE}. It indicates that under certain assumptions, such as the existence and uniqueness of solutions to the original backward equations, mapping $\Phi$ has a unique fixed point. In fact, the assumptions on terminal condition $\xi$ and generator $f$ in Proposition \ref{prop:BML-interpretation-nonlinear-decoupled-FBSDE} are already sufficient. In this case, it can be proven that there exists a norm $\|\cdot\|_\beta$ on $\mathcal{M}[0,T]$ such that mapping $\Phi$ is a strict contraction relative to this metric \cite[Theorem 7.3.2]{yong_zhou_1999}. Therefore, according to the fixed-point theorem, $\Phi$ has a unique fixed point corresponding to the solution $(Y,Z)$ of the backward equations.

By using the definition of mapping $\Phi$ in the proposition, we can replace $(\widetilde{Y},\widetilde{Z})$ with $\Phi(\tilde{y},\tilde{z})$ in the
conclusion of Proposition \ref{prop:BML-interpretation-nonlinear-decoupled-FBSDE}, which gives us equation \eqref{eq:indirect-eq-dist-BML} mentioned earlier in Section \ref{ssec:simple-decoupled-FBSDE}. From this perspective, the BML value of the trial solution $(\tilde{y},\tilde{z})$ is the distance between it and $\Phi(\tilde{y},\tilde{z})$! Once this distance is minimized to zero, it is equivalent to finding a fixed point of $\Phi$, which means we have found the solution to the original backward equations.

\begin{theorem}
  \label{th:BML-interpretation-nonlinear-decoupled-FBSDE}
Let $\mu$ be a $\sigma$-finite measure on $[0,T]$.
Consider the backward equations in FBSDE \eqref{eq:nonlinear-decoupled-FBSDE},
and let $(Y,Z)$ denote its solution.
Assume terminal condition $\xi$ and generator $f$ satisfy the
assumptions of Proposition
\ref{prop:BML-interpretation-nonlinear-decoupled-FBSDE}. If
$(\tilde{y}^*,\tilde{z}^*)\in\mathcal{M}[0,T]$ is a minimizer of the
optimization problem
\begin{equation}
  \label{op:BML-mu}
  \min_{\tilde{y},\tilde{z}}\quad
  \operatorname{BML}(\tilde{y},\tilde{z};\mu),
\end{equation}
then it holds that
$$ \operatorname{BML}(\tilde{y}^*,\tilde{z}^*;\mu) =
\operatorname{dist}_\mu((\tilde{y}^*,\tilde{z}^*), (Y,Z)) = 0$$
with
$\operatorname{dist}_\mu$ defined by equation \eqref{eqdef:dist-yz-YZ}.
\end{theorem}

\begin{remark}
This theorem provides the theoretical foundation for using optimization methods to solve decoupled FBSDEs. It states that if a minimizer $(\tilde{y},\tilde{z})$ of optimization problem \eqref{op:BML-mu} is found, then $(\tilde{y},\tilde{z})$ is the solution to the original equation. For fully coupled FBSDEs, please refer to the generalized version of this theorem, which is presented in the next subsection as Theorem \ref{th:BML-interpretation-fully-coupled-nonlinear-FBSDE}.
\end{remark}

Before presenting the proof of Theorem~\ref{th:BML-interpretation-nonlinear-decoupled-FBSDE},
we first establish the following lemma.

\begin{lemma}
    \label{lem:BML-interpretation-nonlinear-decoupled-FBSDE}
    Under the assumptions of
    Theorem~\ref{th:BML-interpretation-nonlinear-decoupled-FBSDE}, if
    $(\tilde{y},\tilde{z})\in\mathcal{M}[0,T]$, then
$$ \operatorname{dist}_\mu((\tilde{y},\tilde{z}),
    \Phi(\tilde{y},\tilde{z})) = 0\quad \iff\quad
    \operatorname{dist}_\mu((\tilde{y},\tilde{z}), (Y,Z)) = 0.$$
\end{lemma}

The basic idea of the proof is to establish the existence and uniqueness of the solution, which is equivalent to proving the uniqueness of a fixed point for mapping $\Phi$. It is foreseeable because if $\Phi$ has multiple fixed points, then the aforementioned equality may not hold since $(\tilde{y},\tilde{z})$ can take on solutions different from $(Y,Z)$.

\begin{proof}[Proof of Lemma \ref{lem:BML-interpretation-nonlinear-decoupled-FBSDE}]
First, we prove the ``$\Leftarrow$'' part. According to Proposition
\ref{prop:BML-interpretation-nonlinear-decoupled-FBSDE}, we have
\begin{equation}
\begin{aligned}
\operatorname{dist}_\mu((\tilde{y},\tilde{z}), \Phi(\tilde{y},\tilde{z})) =& \operatorname{BML}(\tilde{y},\tilde{z};\mu) \\
=& \int_0^T \operatorname{\mathbb{E}} \left|\tilde{y}_t - \biggl(\xi + \int_t^T
f(s,\tilde{y}_s,\tilde{z}_s)\,ds - \int_t^T \tilde{z}_s\,dW_s\biggr)
\right|^2\mu(dt) \\
\leq& 2 \operatorname{\mathbb{E}}\int_0^T \left|Y_t - \biggl(\xi + \int_t^T
f(s,\tilde{y}_s,\tilde{z}_s)\,ds - \int_t^T \tilde{z}_s\,dW_s\biggr)
\right|^2\mu(dt) \\
&+ 2 \operatorname{\mathbb{E}}\int_0^T |\tilde{y}_t-Y_t|^2\,\mu(dt).
\end{aligned}
\label{eq:proof-tmp-estimation101}
\end{equation}
For the second term on the right-hand side of inequality
\eqref{eq:proof-tmp-estimation101} and according to the assumption, we have
\begin{align*}
\operatorname{\mathbb{E}}\int_0^T|\tilde{y}_t-Y_t|^2\,\mu(dt) \leq \operatorname{\mathbb{E}}\int_0^T\left\{|\tilde{y}_t-Y_t|^2 + \int_t^T|\tilde{z}_s-Z_s|^2\,ds \right\}\,\mu(dt) = \operatorname{dist}_\mu((\tilde{y},\tilde{z}), (Y,Z)) = 0.
\end{align*}
Similarly, we have
$$ \operatorname{\mathbb{E}}\int_0^T\int_t^T|\tilde{z}_s - Z_s|^2\,ds\,\mu(dt) = 0.$$
For the first term on the right-hand side of inequality
\eqref{eq:proof-tmp-estimation101}, we can estimate it using the BSDE satisfied
by $Y$
\begin{align*}
& \operatorname{\mathbb{E}}\int_0^T \left|Y_t - \biggl(\xi +
  \int_t^Tf(s,\tilde{y}_s,\tilde{z}_s)\,ds -
  \int_t^T\tilde{z}_s\,dW_s\biggr) \right|^2\mu(dt)\\ = &\,
  \operatorname{\mathbb{E}}\int_0^T \left|\biggl(\xi +
  \int_t^Tf(s,Y_s,Z_s)\,ds
  -\int_t^TZ_s\,dW_s\biggr)\right.\\ &\phantom{2
    \operatorname{\mathbb{E}}\int_0^T\biggl|}-\biggl(\xi +
  \left.\int_t^Tf(s,\tilde{y}_s,\tilde{z}_s)\,ds
  -\int_t^T\tilde{z}_s\,dW_s\biggr)\right|^2\,\mu(dt)\\ \leq&\,2\operatorname{\mathbb{E}}\int_0^T\biggl\{
  \int_t^T[f(s,Y_s,Z_s) - f(s,\tilde{y}_s,\tilde{z}_s)]\,ds
  \biggr\}^2\,\mu(dt) +
  2\operatorname{\mathbb{E}}\int_0^T\biggl\{\int_t^T[Z_s-\tilde{z}_s]\,dW_s\biggr\}^2\,\mu(dt)\\ \leq&\,2L_f^2\operatorname{\mathbb{E}}\int_0^T
  \int_t^T\bigl(|Y_s-\tilde{y}_s|+|Z_s-\tilde{z}_s|\bigr)^2\,ds
  \,\mu(dt) \\ \leq&\,4L_f^2
  \operatorname{\mathbb{E}}\int_0^T\int_t^T\bigl(|Y_s-\tilde{y}_s|^2 +
  |Z_s-\tilde{z}_s|^2\bigr)\,ds\,\mu(dt)\\ =&\,0.
\end{align*}

Then, we proceed to prove the "$\Rightarrow$" part. Let
$(\widetilde{Y},\widetilde{Z})=\Phi(\tilde{y},\tilde{z})$, and define
$$ \overline{Y}=Y-\widetilde{Y},\quad\overline{Z}=Z-\widetilde{Z},\quad
\bar{f}_s=f(s,Y_s,Z_s)-f(s,\tilde{y}_s,\tilde{z}_s). $$
We have the following differential relation
$$d\overline{Y}_t = -\bar{f}_s\,dt + \overline{Z}_t\,dW_t.$$
Let $\gamma=1+8L_f^2$. Applying It\^o's formula to $e^{\gamma s}|\overline{Y}_s|^2$, we obtain
$$ de^{\gamma s}|\overline{Y}_t|^2 = \bigl[ \gamma e^{\gamma
      t}|\overline{Y}_t|^2 + e^{\gamma t}|\overline{Z}_t|^2-\langle
    \bar{f}_t,2e^{\gamma t}\overline{Y}_t\rangle \bigr]\,dt +
\langle2e^{\gamma t}\overline{Y}_t,\overline{Z}_t\,dW_t\rangle. $$
Integrating over the interval $[t,T]$ and taking into account that
$\overline{Y}_T=0$ and the expected value of the stochastic integral
$\int_0^T\langle2e^{\gamma t}\overline{Y}_t,\overline{Z}_t\,dW_t\rangle$
is zero, we have
\begin{align*}
 \operatorname{\mathbb{E}}e^{\gamma t}|\overline{Y}_t|^2 +
 \operatorname{\mathbb{E}}\int_t^T[\gamma e^{\gamma
     s}|\overline{Y}_s|^2+e^{\gamma s}|\overline{Z}_s|^2]\,ds &=
 \operatorname{\mathbb{E}}\int_t^T\langle \bar{f}_s,2e^{\gamma
   s}\overline{Y}_s\rangle\,ds\\ &\leq 2L_f
 \operatorname{\mathbb{E}}\int_t^T e^{\gamma
   s}|\overline{Y}_s|\bigl(|Y_s-\tilde{y}_s|+|Z_s-\tilde{z}_s|\bigr)\,ds\\ &\leq
 2L_f \operatorname{\mathbb{E}}\int_t^T e^{\gamma s}\cdot
 4L_f|\overline{Y}_s|^2\,ds\\ &\hphantom{\leq~}+2L_f
 \operatorname{\mathbb{E}}\int_t^Te^{\gamma s}\cdot
 \frac{1}{16L_f}\bigl(|Y_s-\tilde{y}_s|+|Z_s-\tilde{z}_s|\bigr)^2\,ds\\ &\leq
 8L_f^2\operatorname{\mathbb{E}}\int_t^T e^{\gamma
   s}|\overline{Y}_s|^2\,ds\\ &\hphantom{\leq~}+
 \frac{1}{4}\operatorname{\mathbb{E}}\int_t^Te^{\gamma
   s}\bigl[|Y_s-\tilde{y}_s|^2+|Z_s-\tilde{z}_s|^2\bigr]\,ds.
\end{align*}
By subtracting $8L_f^2\operatorname{\mathbb{E}}\int_t^T e^{\gamma s}|\overline{Y}_s|^2\,ds$
from both sides and considering that $\gamma=1+8L_f^2$, we obtain
\begin{equation}
\label{eq:proof-tmp-303}
\begin{aligned}
\operatorname{\mathbb{E}} e^{\gamma
  t}|Y_t-\widetilde{Y}_t|^2+&\operatorname{\mathbb{E}}\int_t^Te^{\gamma
  s}\bigl[|Y_s-\widetilde{Y}_s|^2 +
  |Z_s-\widetilde{Z}_s|^2\bigr]\,ds\\ & \leq \frac{1}{4}
\operatorname{\mathbb{E}}\int_t^Te^{\gamma
  s}\bigl[|Y_s-\tilde{y}_s|^2+|Z_s-\tilde{z}_s|^2\bigr]\,ds.
\end{aligned}
\end{equation}

On the other hand, from the assumption
$\operatorname{dist}_\mu((\tilde{y},\tilde{z}),(\widetilde{Y},\widetilde{Z}))=0$,
we have
$$ \operatorname{\mathbb{E}}\int_0^T\int_0^T|\tilde{y}_s-\widetilde{Y}_s|^2\,ds\,\mu(dt)=0,\qquad
\operatorname{\mathbb{E}}\int_0^T\int_t^T|\tilde{z}_s-\widetilde{Z}_s|^2\,ds\,\mu(dt)=0. $$
Hence, integrating both sides of inequality~\eqref{eq:proof-tmp-303}, we can further simplify it as follows
\begin{equation*}
\begin{aligned}
\operatorname{\mathbb{E}} \int_0^Te^{\gamma
  t}|Y_t-\widetilde{Y}_t|^2\,\mu(dt)+&\operatorname{\mathbb{E}}\int_0^T\int_t^Te^{\gamma
  s}\bigl[|Y_s-\widetilde{Y}_s|^2 +
  |Z_s-\widetilde{Z}_s|^2\bigr]\,ds\,\mu(dt)\\ & \leq \frac{1}{2}
\operatorname{\mathbb{E}}\int_0^T\int_t^Te^{\gamma
  s}\bigl[|Y_s-\widetilde{Y}_s|^2+|Z_s-\widetilde{Z}_s|^2\bigr]\,ds\,\mu(dt).
\end{aligned}
\end{equation*}
Thus, we have
$$ \operatorname{\mathbb{E}} \int_0^Te^{\gamma
  t}|Y_t-\widetilde{Y}_t|^2\,\mu(dt) = 0,\quad
\operatorname{\mathbb{E}}\int_0^T\int_t^Te^{\gamma
  s}|Z_s-\widetilde{Z}_s|^2\,ds\,\mu(dt) = 0. $$
This is equivalent to
$$ \operatorname{\mathbb{E}} \int_0^T|Y_t-\widetilde{Y}_t|^2\,\mu(dt)
= 0,\quad \operatorname{\mathbb{E}}\int_0^T\int_t^T|Z_s-\widetilde{Z}_s|^2\,ds\,\mu(dt)
= 0.$$
Adding these equations gives
$\operatorname{dist}_\mu((\widetilde{Y},\widetilde{Z}),(Y,Z))=0$.

In conclusion, the lemma statement holds true.
\end{proof}

\begin{proof}[Proof of Theorem~\ref{th:BML-interpretation-nonlinear-decoupled-FBSDE}]
First, according to the lemma above, the BML of the true solution must be
zero, since $\operatorname{dist}_\mu((Y,Z), \Phi(Y,Z))=0$. By
Proposition~\ref{prop:BML-interpretation-nonlinear-decoupled-FBSDE}, we have
$\operatorname{BML}(Y,Z;\mu)=0$. Therefore, $(Y,Z)$ is a minimizer of the
optimization problem~\eqref{op:BML-mu}. 

On the other hand, because $(\tilde{y}^*,\tilde{z}^*)$ is also a minimizer, we have
$$ \operatorname{BML}(\tilde{y}^*,\tilde{z}^*;\mu) \leq
  \operatorname{BML}(Y,Z;\mu)=0.$$
This implies that the BML value of $(\tilde{y}^*,\tilde{z}^*)$ is also zero.
Applying Proposition~\ref{prop:BML-interpretation-nonlinear-decoupled-FBSDE} and
Lemma~\ref{lem:BML-interpretation-nonlinear-decoupled-FBSDE} again, we obtain
$$0=\operatorname{BML}(\tilde{y}^*,\tilde{z}^*;\mu)=\operatorname{dist}_\mu((\tilde{y}^*,\tilde{z}^*),(Y,Z)).$$
Hence, the theorem is proved.
\end{proof}

Before concluding this section, we emphasize the difference between optimizing the BML value and fixed-point iteration. Proposition~\ref{prop:BML-interpretation-nonlinear-decoupled-FBSDE} shows that the BML value of a trial solution $(\tilde{y},\tilde{z})$ is equal to its distance from the next point $\Phi(\tilde{y},\tilde{z})$ in the fixed-point iteration. Furthermore, Theorem~\ref{th:BML-interpretation-nonlinear-decoupled-FBSDE} states that when the BML value is zero, the trial solution coincides with the true solution. Therefore, optimizing BML can be viewed as a method for solving the FBSDE.

On the other hand, the theoretical foundation of fixed-point iteration lies in the fact that the distance between a trial solution and its mapping
$\Phi(\tilde{y},\tilde{z})$ will decrease as long as $\Phi$ is a contraction mapping. Under the standard assumptions (as given in the hypothesis of Theorem~\ref{th:BML-interpretation-nonlinear-decoupled-FBSDE}), the contraction property of mapping $\Phi$ is guaranteed. Therefore, by iterating the scheme
$$ 
(\tilde{y}^{(k+1)},\tilde{z}^{(k+1)}) = \Phi(\tilde{y}^{(k)},\tilde{z}^{(k)}), 
$$
it is ensured that as $k$ tends to infinity, $(\tilde{y}^{(k+1)},\tilde{z}^{(k+1)})$ tends to a fixed point of $\Phi$.

However, the fixed-point iteration explicitly applies the mapping $\Phi$, meaning that we need to solve a BSDE at each iteration step, even if it is a simple linear equation like \eqref{eq:BML-linearized-BSDE}. On the other hand, if we choose to optimize the BML, we can avoid this. The computation of BML value is based on definition \eqref{eqdef:fully-coupled-nonlinear-BML}, which only requires substituting the trial solution $(\tilde{y},\tilde{z})$ into it. By optimizing the BML, we bypass the need to solve $\Phi(\tilde{y}^{(k)},\tilde{z}^{(k)})$ at each iteration step and instead choose to optimize $\operatorname{dist}_\mu((\tilde{y}^{(k)},\tilde{z}^{(k)}),\Phi(\tilde{y}^{(k)},\tilde{z}^{(k)}))$. In other words, this can be understood as solving the fixed-point equation $x=f(x)$ and choosing to optimize $|x-f(x)|$.

\subsection{General Coupled FBSDE}

In this subsection, we consider equation \eqref{eq:fully-coupled-nonlinear-FBSDE} itself, namely, the coupled nonlinear
FBSDE, where the forward equation depends on the solution of the backward equation, and vice
versa. Unlike the special case considered in the previous subsection, 
it is not possible to simply consider the forward and backward equations
separately for the coupled nonlinear FBSDE. However, given a trial solution
$(\tilde{y},\tilde{z})$ for the backward equation, we can still substitute it
into the position of $(Y,Z)$ in the forward equation, obtaining SDE
\eqref{eq:BML-tilde-X}. This SDE no longer contains the solution $(X,Y,Z)$ of
the original equations, so we can apply the Euler-Maruyama discretization scheme
to solve it and obtain the stochastic process $\widetilde{X}$. Then, we
substitute $\widetilde{X}$ into the position of $X$ in the backward equation,
resulting in a new \textit{linear} BSDE. This new BSDE also no longer contains the
solution $(X,Y,Z)$ of the original equations. In other words, given a trial
solution $(\tilde{y},\tilde{z})\in\mathcal{M}[0,T]$, we introduce the following
decoupled simplified FBSDE
\begin{equation}
\label{eq:manually-decoupled-nonlinear-FBSDE}
\left\{
\begin{aligned}
\widetilde{X}_t &= x_0 +
\int_0^tb(s,\widetilde{X}_s,\tilde{y}_s,\tilde{z}_s)\,ds +
\int_0^t\sigma(s,\widetilde{X}_s,\tilde{y}_s,\tilde{z}_s)\,dW_s,\\ \widetilde{Y}_t
&= g(\widetilde{X}_T) +
\int_t^Tf(s,\widetilde{X}_s,\tilde{y}_s,\tilde{z}_s)\,ds -
\int_t^T\widetilde{Z}_s\,dW_s,
\end{aligned}
\right.
\end{equation}
and its solution is denoted as $(\widetilde{X},\widetilde{Y},\widetilde{Z})$.
Applying Theorem~\ref{th:BML-interpretation-simple-decoupled-FBSDE} to this
equation, we obtain a result similar to
Proposition~\ref{prop:BML-interpretation-nonlinear-decoupled-FBSDE}.

\begin{proposition}
\label{prop:BML-interpretation-fully-coupled-nonlinear-FBSDE}
Let $\mu$ be a $\sigma$-finite measure on $[0,T]$.
Consider the coupled nonlinear FBSDE \eqref{eq:fully-coupled-nonlinear-FBSDE}.
Assume $b,\sigma,f\in
L^2_{\mathcal{F}}(0,T;W^{1,\infty}(M;N))$ and $g\in
L^2_{\mathcal{F}}(\Omega;W^{1,\infty}(M;N))$, where $M,N$ are Euclidean spaces
of appropriate dimensions.
Then, for any trial solution $(\tilde{y},\tilde{z})\in \mathcal{M}[0,T]$, its
BML value is equal to the distance between $(\tilde{y},\tilde{z})$ and the
solution $(\widetilde{Y},\widetilde{Z})$ of the decoupled simplified FBSDE
\eqref{eq:manually-decoupled-nonlinear-FBSDE}, i.e.,
$$\operatorname{BML}(\tilde{y},\tilde{z};\mu) =
\operatorname{dist}_\mu((\tilde{y},\tilde{z}), (\widetilde{Y},\widetilde{Z})).$$
Here, $\operatorname{dist}_\mu$ is defined by \eqref{eqdef:dist-yz-YZ}.
\end{proposition}

Similar to Proposition \ref{prop:BML-interpretation-nonlinear-decoupled-FBSDE}
in the previous subsection, the key to proving this proposition is to realize
that for the same trial solution $(\tilde{y},\tilde{z})$, its BML value with
respect to the decoupled simplified FBSDE
\eqref{eq:manually-decoupled-nonlinear-FBSDE} is equal to that value with
respect to the coupled nonlinear FBSDE \eqref{eq:fully-coupled-nonlinear-FBSDE}.

\begin{proof}
Let $(\widetilde{X},\widetilde{Y},\widetilde{Z})$ be the solution to the
decoupled simplified FBSDE \eqref{eq:manually-decoupled-nonlinear-FBSDE}. For
any trial solution $(\tilde{y},\tilde{z})\in \mathcal{M}[0,T]$, let
$\tilde{\xi}=g(\widetilde{X}_T)$ and
$\tilde{f}_s=f(s,\widetilde{X}_s,\tilde{y}_s,\tilde{z}_s)$. Then,
$(\widetilde{Y},\widetilde{Z})$ satisfies the BSDE
\begin{equation}
\widetilde{Y}_t = \tilde{\xi} + \int_t^T\tilde{f}_s\,ds - \int_t^T\widetilde{Z}_s\,dW_s.
\end{equation}
By the assumptions on terminal condition $g$ and generator $f$, it can
be easily shown that this BSDE has a unique solution. Moreover, noting that this
equation satisfies the conditions of Theorem
\ref{th:BML-interpretation-simple-decoupled-FBSDE}, we have
\begin{equation}
\label{eq:proof-tmp-622}
\operatorname{dist}_\mu((\tilde{y},\tilde{z}), (\widetilde{Y},\widetilde{Z})) =
\int_0^T \operatorname{\mathbb{E}}\left|\tilde{y}_t - \biggl(\tilde{\xi} +
\int_s^T\tilde{f}_s\,ds - \int_t^T\tilde{z}_s\,dW_s\biggr)\right|^2\,\mu(dt).
\end{equation}
On the other hand, noting that the solution to the forward equation of FBSDE
\eqref{eq:manually-decoupled-nonlinear-FBSDE} satisfies stochastic
differential equation \eqref{eq:BML-tilde-X}, we have, according to the
definition of BML value,
\begin{equation}
\label{eq:proof-tmp-623}
\operatorname{BML}((\tilde{y},\tilde{z};\mu)) = \int_0^T
\operatorname{\mathbb{E}}\left|\tilde{y}_t - \biggl(g(\widetilde{X}_T) +
\int_s^T{f}(s,\widetilde{X}_s,\tilde{y}_s,\tilde{z}_s)\,ds -
\int_t^T\tilde{z}_s\,dW_s\biggr)\right|^2\,\mu(dt).
\end{equation}
Since the right-hand side of equation \eqref{eq:proof-tmp-622} is equal to the
right-hand side of equation \eqref{eq:proof-tmp-623}, we conclude that
$$\operatorname{BML}(\tilde{y},\tilde{z};\mu) =
\operatorname{dist}_\mu((\tilde{y},\tilde{z}), (\widetilde{Y},\widetilde{Z})).$$
Thus, the proposition is proved.
\end{proof}

Similar to the discussion provided in the previous subsection, the BML value for
coupled nonlinear FBSDE cases can also be interpreted from a fixed-point iteration
perspective. When the assumptions of Proposition
\ref{prop:BML-interpretation-fully-coupled-nonlinear-FBSDE} hold, for any pair
of stochastic processes $(\tilde{y},\tilde{z})\in\mathcal{M}[0,T]$, the forward
equation in FBSDE \eqref{eq:manually-decoupled-nonlinear-FBSDE} has a unique
strong solution $\widetilde{X}\in
L^2_{\mathcal{F}}(\Omega;C[0,T];\mathbb{R}^n)$. Thus, the backward equation in
FBSDE \eqref{eq:manually-decoupled-nonlinear-FBSDE} has a unique solution
$(\widetilde{Y},\widetilde{Z})$. This defines a map $\Phi$ on the space
$\mathcal{M}[0,T]$, which maps $(\tilde{y},\tilde{z})$ to the unique solution
$(\widetilde{Y},\widetilde{Z})$ of FBSDE
\eqref{eq:manually-decoupled-nonlinear-FBSDE}. Similarly, the fixed-points of
this map $\Phi$ are equivalent to the solutions of the coupled nonlinear FBSDE
\eqref{eq:fully-coupled-nonlinear-FBSDE}. On the other hand,
$\operatorname{dist}_\mu((\tilde{y},\tilde{z}), (\widetilde{Y},\widetilde{Z}))$
in the conclusion of Proposition
\ref{prop:BML-interpretation-fully-coupled-nonlinear-FBSDE} can also be replaced
by $\operatorname{dist}_\mu((\tilde{y},\tilde{z}),\Phi(\tilde{y},\tilde{z}))$.
Therefore, optimizing the BML is equivalent to optimizing the distance
between the trial solution and $\Phi(\tilde{y},\tilde{z})$. It can be expected
that when this distance is optimized to zero, the trial solution coincides with
the true solution.

\begin{theorem}
  \label{th:BML-interpretation-fully-coupled-nonlinear-FBSDE}
Let $\mu$ be a $\sigma$-finite measure on $[0,T]$ and consider the coupled nonlinear FBSDE \eqref{eq:fully-coupled-nonlinear-FBSDE}
and let $(X,Y,Z)$ be its solution. Suppose that the assumptions of Theorem
\ref{th:uniqueness-and-existence-fully-coupled-nonlinear-FBSDE} hold for
$b,\sigma,f,g$.
If
$(\tilde{y}^*,\tilde{z}^*)\in\mathcal{M}[0,T]$ is a minimizer of the
optimization problem
\begin{equation}
  \label{op:BML-mu-alias}
  \min_{\tilde{y},\tilde{z}}\quad
  \operatorname{BML}(\tilde{y},\tilde{z};\mu),
\end{equation}
then it achieves zero, i.e.,
$\operatorname{BML}(\tilde{y}^*,\tilde{z}^*;\mu) =0$.
In particular,
\begin{equation}
\label{eq:weak-solution}
\left\{
\begin{aligned}
{x}^*_t &= x_0 +
\int_0^tb(s,{x}^*_s,{y}^*_s,{z}^*_s)\,ds +
\int_0^t\sigma(s,{x}^*_s,{y}^*_s,{z}^*_s)\,dW_s,\\
{y}^*_t
&= g({x}^*_T) +
\int_t^Tf(s,{x}^*_s,{y}^*_s,{z}^*_s)\,ds -
\int_t^T{z}^*_s\,dW_s
\end{aligned}
\right.
\end{equation}
holds $d\mu\otimes d\,\mathbb{P}$-a.e..
Moreover, if $\mu$ is equivalent to the Lebesgue
measure, then $\operatorname{dist}_\mu((y^*,z^*),(Y,Z))=0$.
Even stronger, if $\mu$ is equivalent to  the Lebesgue
measure,
then for all $t\in[0,T]$,
Eq.~\eqref{eq:weak-solution} holds almost surely,
which means $(y^*,z^*)$ solves the coupled nonlinear FBSDE \eqref{eq:fully-coupled-nonlinear-FBSDE}.
\end{theorem}

\begin{proof}
This proof is easier than the decoupled case as we
do not try to prove
Lemma~\ref{lem:BML-interpretation-nonlinear-decoupled-FBSDE}.
In fact, the solution $(X,Y,Z)$ indeed exists by
Theorem~\ref{th:uniqueness-and-existence-fully-coupled-nonlinear-FBSDE}.
Hence, $\operatorname{BML}(Y,Z;\mu)=0$ is true by definition,
which means the minimum of the optimization problem is 0.
Therefore, $\operatorname{BML}(y^*,z^*;\mu)=0$,
and Eq.~\eqref{eq:weak-solution} holds $d\mu\otimes d\,\mathbb{P}$-a.e.
by the definition of BML;
see Eq.~\eqref{eqdef:fully-coupled-nonlinear-BML} and
Eq.~\eqref{eq:BML-tilde-X}.

Now suppose $\mu$ is equivalent to the Lebesgue measure.
Then, Eq.~\eqref{eq:weak-solution} holds $dt\otimes d\,\mathbb{P}$-a.e..
In particular,
$$
y_t - g({x}^*_T) -
\int_t^Tf(s,{x}^*_s,{y}^*_s,{z}^*_s)\,ds +
\int_t^T{z}^*_s\,dW_s = 0
$$
holds $dt\otimes d\,\mathbb{P}$-a.e..
Noting that the left-hand side is continuous in time,
it must hold almost surely for all $t\in[0,T]$.
Hence, Eq.~\eqref{eq:weak-solution} tells us that 
$(y^*,z^*)$ is a solution to the coupled nonlinear FBSDE.
Again, by Theorem~\ref{th:uniqueness-and-existence-fully-coupled-nonlinear-FBSDE},
the solution is unique,
so $\operatorname{dist}_\mu((y^*,z^*),(Y,Z))=0$.
This completes the proof.
\end{proof}

\begin{remark}
We should emphasize that when degenerated to the decoupled case, Theorem~\ref{th:BML-interpretation-nonlinear-decoupled-FBSDE}
is more general on the choice of $\mu$. More specifically, as long as the BML value of $(y^*,z^*)$ is zero, Theorem~\ref{th:BML-interpretation-nonlinear-decoupled-FBSDE} immediately concludes that $\operatorname{dist}_\mu((y^*,z^*),(Y,Z))=0$, without assuming $\mu$ is equivalent to the Lebesgue measure. In other words, the conclusion obtained in the decoupled case does not hold in the fully coupled case. This difference is, in fact, a direct consequence of the failure of Lemma~\ref{lem:BML-interpretation-nonlinear-decoupled-FBSDE} in the coupled case. Nevertheless, $\mu$ is not a part of the FBSDE and is only introduced manually to define the BML value. We are free to choose it to fulfill the assumption.
\end{remark}

\section{Discretization and Optimization Algorithm Design}

Based on the analysis in the previous sections, solving the coupled FBSDE can be transformed into optimization problem \eqref{op:ell-yz}, where the error function $\ell$ can be chosen as the BML value of the trial solution. By employing optimization methods, we can find a stochastic process
$(\tilde{y}^*,\tilde{z}^*)$ that minimizes the BML value and is considered as an approximate solution to the backward equation in the coupled FBSDE. Finally, by substituting the obtained approximate solution $(\tilde{y}^*,\tilde{z}^*)$ into the forward equation and applying the standard Euler-Maruyama scheme for time discretization, we can calculate the approximate solution to the forward process. This section will introduce the calculation method for the BML value and provide specific numerical algorithms.

\subsection{Parameterization and Sampling of the Trial Solution}

Firstly, it should be noted that the optimization variables in problem
\eqref{op:ell-yz} consist of a pair of stochastic processes
$\{(\tilde{y}_s,\tilde{z}_s); 0\leq s\leq T\}$. In order to apply numerical
optimization algorithms, we first parameterize the stochastic processes, and this 
allows us to transform the search space of the optimization problem into a
finite-dimensional space. Considering the Markov property of FBSDEs, the
solution to the backward equation can be generally represented as
\begin{equation} \label{eq:true-uv} Y_t=v(t,X_t), \quad Z_t=u(t,X_t), \quad
t\in[0,T], 
\end{equation}
where the nonlinear Feynman-Kac lemma in stochastic
analysis provides such a form of solution for a class of FBSDEs, with the
bivariate function $v(\cdot,\cdot)$ satisfying a second-order parabolic partial
differential equation and function $u$ satisfying
$u(\cdot,\cdot)=\sigma^\intercal\partial_x v(\cdot,\cdot)$. Similarly, we select the 
parameterized functions $v^\theta$ and $u^\theta$, which in this paper are neural
networks, to replace the roles of $v$ and $u$.

The method for constructing the trial solution $\tilde{y}_s, \tilde{z}_s$
through parameterized functions is as follows. Let $\Theta$ be the set of
parameters. For $\theta\in\Theta$, denote the corresponding parameterized
functions as $v^\theta$ and $u^\theta$. Calculate the approximate solution to
the forward equation
\begin{equation} \label{eq:tilde-X-theta}
\widetilde{X}^\theta_t = x_0 + \int_0^t b^\theta(s,\widetilde{X}^\theta_s)ds +
\int_0^t \sigma^\theta(s,\widetilde{X}^\theta_s)dW_s, \quad t\in[0,T],
\end{equation}
where
$$
b^\theta(t,x)\coloneqq
b(t,x,v^\theta(t,x),u^\theta(t,x)), \quad \sigma^\theta(t,x)\coloneqq
\sigma(t,x,v^\theta(t,x),u^\theta(t,x)).
$$
Then, define
\begin{equation}
\label{eq:tilde-yz-theta} \tilde{y}^\theta_t=v^\theta(t,\widetilde{X}^\theta_t),
\quad \tilde{z}^\theta_t=u^\theta(t,\widetilde{X}^\theta_t), \quad t\in[0,T].
\end{equation}

The sampling method for representing the trial solution
$(\tilde{y}^\theta,\tilde{z}^\theta)$ is also straightforward. Firstly, sample a
path of standard Brownian motion using a random number generator, for example,
$\{W_t(\omega_r); 0\leq t\leq T\}$. Then, substitute this sample path into
equation \eqref{eq:tilde-X-theta} to obtain the corresponding sample path of the
approximate forward process, denoted as $\{\widetilde{X}^\theta_t(\omega_r);
0\leq t\leq T\}$. Finally, by utilizing the representation
\eqref{eq:tilde-yz-theta}, we obtain the sample path
$\{(\tilde{y}^\theta_t(\omega_r),\tilde{z}^\theta_t(\omega_r)); 0\leq t\leq T\}$
of the trial solution.

\subsection{Calculating BML Value}\label{ssec:calc-BML}

\begin{algorithm}
  \caption{Numerical Algorithm for Solving FBSDE via Optimization of BML.\label{alg:solve-FBSDE-via-BML}}
\small
\hspace*{\algorithmicindent} \textbf{Input:} FBSDE parameters $x_0,b,\sigma,f,g,T$. \\
\hspace*{\algorithmicindent} \textbf{Output:} Set of sample paths of the solution \\
\hspace*{\algorithmicindent} \hspace{1.6em} $\{(X_{t_i}(\omega_r), Y_{t_i}(\omega_r), Z_{t_i}(\omega_r))\,|\,0\leq i\leq N,0\leq r < M\}$. \\
\hspace*{\algorithmicindent} \textbf{Parameters:} Number of time intervals $N$, number of samples $M$, gradient update step size $\eta$, tolerance error $\epsilon$, measure $\mu$ \\
\vspace*{-5pt}
\begin{algorithmic}[1]
\STATE Construct neural networks $\tilde{y}^\theta$ and $\tilde{z}^\theta$, initialize the weights $\theta$;
\STATE \texttt{\#} { \it Sample $M$ paths of standard Brownian motion with time interval $\Delta t=T/N$;} \label{state:sample-W}
\FOR{ $r=0$ \TO $M-1$}
\STATE $W_{t_0}(\omega_r)\leftarrow 0$
\FOR{ $i=0$ \TO $N-1$}
\STATE Sample $\epsilon_{i,r}$ from a normal distribution with mean zero and variance $\frac{T}{N}$;
\STATE $W_{t_{i+1}}(\omega_r) \leftarrow W_{t_{i}}(\omega_r) + \epsilon_{i,r}$;
\ENDFOR
\ENDFOR
\STATE \texttt{\#} { \it Compute $M$ corresponding sample paths of the forward process;}
\FOR{ $r=0$ \TO $M-1$}
\STATE $X_{t_0}(\omega_r)\leftarrow x_0$;
\STATE $Y_{t_0}(\omega_r)\leftarrow \tilde{y}^\theta(t_0,x_0)$;
\STATE $Z_{t_0}(\omega_r)\leftarrow \tilde{z}^\theta(t_0,x_0)$;
\FOR{ $i=0$ \TO $N-1$}
\STATE Compute $X_{t_{i+1}}(\omega_r)$ according to \eqref{eq:discrete-X-theta};
\STATE $Y_{t_i}(\omega_r)\leftarrow \tilde{y}^\theta(t_i,X_{t_{i+1}}(\omega_r))$;
\STATE $Z_{t_i}(\omega_r)\leftarrow \tilde{z}^\theta(t_i,X_{t_{i+1}}(\omega_r))$;
\ENDFOR
\ENDFOR
\STATE \texttt{\#} { \it Estimate the BML value;}
\IF{$\mu(dt)=\delta(dt)$}
\STATE $\ell(\theta)\leftarrow\delta$-BML, as given in \eqref{eq:calc-delta-BML};
\ENDIF
\IF{$\mu(dt)\propto dt$}
\STATE $\ell(\theta)\leftarrow\lambda$-BML, as given in \eqref{eq:calc-lambda-BML};
\ENDIF
\IF{$\mu(dt)\propto e^{-\gamma t}\,dt$}
\STATE $\ell(\theta)\leftarrow\gamma$-BML, as given in \eqref{eq:calc-gamma-BML};
\ENDIF
\STATE \texttt{\#} { \it Check if the accuracy is sufficient;}
\IF{$|\ell(\theta)| < \epsilon$}
\STATE \texttt{\#} { \it  Output the results;}
\RETURN $\{(X_{t_i}(\omega_r), Y_{t_i}(\omega_r), Z_{t_i}(\omega_r))\,|\,0\leq i\leq N,0\leq r < M\}$;
\ELSE
\STATE \texttt{\#} { \it  Perform one gradient update on $\theta$;}
\STATE $\theta\leftarrow \theta - \eta\nabla\ell(\theta)$;
\STATE Go back to Line \ref{state:sample-W};
\ENDIF
\end{algorithmic}
\end{algorithm}

Using the method described in the previous section to parameterize the
experimental solution, the optimization problem to be solved can now be written
as
$$
\min_{\theta\in\Theta}\quad\operatorname{BML}(\theta;\mu)\coloneqq\operatorname{BML}(\tilde{y}^\theta,\tilde{z}^\theta;\mu).
$$
Next, we will explain the specific calculation of the BML value given a
certain $\theta$.

According to the definition in
Equation~\eqref{eqdef:fully-coupled-nonlinear-BML}, the calculation of the BML
value mainly involves computing multiple integrals and expectations of the
stochastic processes $(\widetilde{X}^\theta,\tilde{y}^\theta,\tilde{z}^\theta)$.
Monte Carlo methods are used to estimate the expectations, and a simple
rectangle formula is used to estimate the time integrals. The basic procedure is
to first calculate all the time integrals for a single sample path, and then
take the average of the results from all samples to estimate the expectation.
For convenience, we choose $N+1$ equidistant time nodes on the interval $[0,T]$
(including both endpoints), denoted as
$$ t_i = \frac{iT}{N}, \qquad
i=0,1,2,\ldots,N.
$$
Therefore, for a single sample path of Brownian motion
$\{W_t(\omega_r),~0\leq t\leq T\}$, the state values at $N+1$ nodes can be
calculated:
\begin{equation} \label{eq:discrete-X-theta}
\widetilde{X}_{t_{i+1}}(\omega_r) = \widetilde{X}_{t_i}(\omega_r) +
b^\theta(t_i,\widetilde{X}_{t_i}(\omega_r))\frac{T}{N} +
\sigma^\theta(t_i,\widetilde{X}_{t_i}(\omega_r))(W_{t_i+1}(\omega_r)-W_{t_i}(\omega_r)),\quad
\widetilde{X}_{t_0}\equiv x_0.
\end{equation}
By applying the representation
given in Equation~\eqref{eq:tilde-yz-theta}, we obtain the values of the
backward equation's experimental solution at $N+1$ nodes.

In this paper, we mainly consider three different cases for measure $\mu$:
the Dirac measure, the normalized Lebesgue measure, and the measure with
exponential decay. The investigation of other types of measures for $\mu$ can be
left for future work. Let $M$ denote the number of samples used in the Monte
Carlo method. The formulas for calculating the BML value in these three cases
are as follows.
\begin{itemize} 
\item When $\mu$ is the Dirac measure centered at the initial
time, the BML value (denoted as $\delta$-BML) is given by 
\begin{equation}\label{eq:calc-delta-BML}
\begin{aligned} \frac{1}{M} \sum_{r=0}^{M-1}
\biggl|\tilde{y}^\theta_{t_0}(\omega_r) -
\biggl(g(\widetilde{X}^\theta_{t_N})(\omega_r) &+ \frac{T}{N}
\sum_{i=0}^{N-1}f(t_i,\widetilde{X}^\theta_{t_i},\tilde{y}^\theta_{t_i},\tilde{z}^\theta_{t_i})(\omega_r)\\
& -
\sum_{i=0}^{N-1}\tilde{z}^\theta_{t_i}(W_{t_{i+1}}-W_{t_i})(\omega_r)\biggr)\biggr|^2.
\end{aligned}
\end{equation} 
\item When $\mu$ is the normalized Lebesgue measure
on the time interval $[0,T]$, the BML value (denoted as $\lambda$-BML) is given
by 
\begin{equation} \label{eq:calc-lambda-BML}
\begin{aligned} \frac{1}{M}
\sum_{r=0}^{M-1} \frac{1}{N}\sum_{j=0}^{N-1}
\biggl|\tilde{y}^\theta_{t_j}(\omega_r) -
\biggl(g(\widetilde{X}^\theta_{t_N})(\omega_r) &+ \frac{T}{N}
\sum_{i=j}^{N-1}f(t_i,\widetilde{X}^\theta_{t_i},\tilde{y}^\theta_{t_i},\tilde{z}^\theta_{t_i})(\omega_r)\\
& -
\sum_{i=0}^{N-1}\tilde{z}^\theta_{t_i}(W_{t_{i+1}}-W_{t_i})(\omega_r)\biggr)\biggr|^2.
\end{aligned}
\end{equation}
\item When $\mu$ is a measure with exponential
decay over time, the BML value (denoted as $\gamma$-BML) is given by
\begin{equation} \label{eq:calc-gamma-BML}
\begin{aligned}
\frac{1}{M}
\sum_{r=0}^{M-1} \frac{1-e^{-\gamma}}{1-e^{-\gamma N}}
\sum_{j=0}^{N-1}e^{-\gamma j} \biggl|\tilde{y}^\theta_{t_j}(\omega_r) -
\biggl(g(\widetilde{X}^\theta_{t_N})(\omega_r) &+ \frac{T}{N}
\sum_{i=j}^{N-1}f(t_i,\widetilde{X}^\theta_{t_i},\tilde{y}^\theta_{t_i},\tilde{z}^\theta_{t_i})(\omega_r)\\
& -
\sum_{i=0}^{N-1}\tilde{z}^\theta_{t_i}(W_{t_{i+1}}-W_{t_i})(\omega_r)\biggr)\biggr|^2.
\end{aligned}
\end{equation}
\end{itemize}

Different BML values correspond to different optimization problems. According to
the theoretical analysis in the previous section, they should all converge to
the parameter values corresponding to the true solution. However, in practical
calculations, errors due to time discretization, parameterization, and
floating-point computations may lead to different results when optimizing
different BML values. The choice of BML value can be made based on specific
needs. This is also the significance of investigating different values for
measure $\mu$. It is worth mentioning that the $\gamma$-BML value with
exponential decay is the most flexible. When $\gamma$ is set to a very small
positive value, the weight assigned to each time interval, $e^{-\gamma j}$, is
approximately 1. Therefore, $\gamma$-BML$\approx\lambda$-BML. On the other hand,
when $\gamma$ is set to a very large positive value, the weight decreases
rapidly over time and may only concentrate on the first time interval, resulting
in $\gamma$-BML$\approx\delta$-BML. In short, by adjusting the value of
$\gamma$, the $\gamma$-BML value can simulate either the $\delta$-BML value or
the $\lambda$-BML value. In this paper, the decay factor $\gamma$ is set to 0.05
to distinguish it from the other two cases.

\subsection{FBSDE Numerical Solution Algorithm based on Optimizing BML Value}

So far, the solution of the FBSDE has been transformed into the optimization
problem of $\operatorname{BML}(\theta;\mu)$, and the previous section has
detailed how to compute the BML value for a given parameter $\theta$. Putting
all these together, we can summarize the algorithm for solving FBSDE~\eqref{%
eq:fully-coupled-nonlinear-FBSDE}
via optimizing the BML value as Algorithm~\ref{alg:solve-FBSDE-via-BML}.

\section{Numerical Experiments}

This section presents specific examples of solving FBSDEs numerically using neural network optimization based on the BML value. The examples tested here include two coupled nonlinear FBSDEs with analytic solutions and a first-order adjoint equation for a high-dimensional stochastic optimal control problem.

Here are the specific settings for neural network optimization. All the experiments are conducted in the PyTorch framework. The functions $\tilde{y}^\theta(t,x)$ and $\tilde{z}^\theta(t,x)$ used to represent the trial solutions are fully connected neural networks, both of which take an $n+1$ dimensional vector consisting of $[t,x]$ as input. $\tilde{y}^\theta$ outputs an $m$-dimensional vector, while $\tilde{z}^\theta$ outputs an $md$-dimensional vector (equivalent to an $m\times d$ matrix). The optimization algorithm used is the Adam optimizer provided by PyTorch. For the detailed settings of each example in the neural network training, please refer to Table~\ref{tab:parameter-settings}.

Considering that the solution of the FBSDE only depends on the value at initial time $t=0$ and is not affected by the random sample point $\omega$, we calculate the relative error of the BML method estimated $Y_0$ using the reference $Y_0$ obtained from the analytic solution or another numerical method. This relative error is used as an evaluation criterion, as shown in Table~\ref{tab:solveFBSDE-results}. Here are the specific descriptions of the examples considered.

\begin{table}
\centering
\caption{Settings of training parameters for different examples.}
\label{tab:parameter-settings}
\begin{tabular}{llcccc}
\toprule
& Objective & Sampling Size M & Number of Layers & Number of Neurons & Learning Rate \\
\hline
Example~\ref{example:FuSinCos-EXP1} & $\delta$-BML & 4096 & 2 & 8 & 0.001 \\ 
& $\lambda$-BML & 4096 & 2 & 8 & 0.001 \\ 
& $\gamma$-BML & 4096 & 2 & 8 & 0.001 \\ \hline
Example~\ref{example:LongSin-EXP1} & $\delta$-BML & 1024 & 3 & 32 & 0.001 \\ 
& $\lambda$-BML & 1024 & 3 & 32 & 0.001 \\ 
& $\gamma$-BML & 1024 & 3 & 32 & 0.001 \\ \hline
Example~\ref{example:JiLQ-EXP1} & $\delta$-BML & 64 & 2 & 16 & 0.001 \\ 
& $\lambda$-BML & 64 & 2 & 16 & 0.001 \\ 
& $\gamma$-BML & 64 & 2 & 16 & 0.001 \\ \hline
Example~\ref{example:JiLQ-EXP2} & $\delta$-BML & 64 & 2 & 16 & 0.0005 \\ 
& $\lambda$-BML & 64 & 2 & 16 & 0.002 \\ 
& $\gamma$-BML & 64 & 2 & 16 & 0.002 \\
\bottomrule
\end{tabular}
\end{table}

\begin{figure}
    \centering
    \subfigure[Training process of Example~\ref{example:FuSinCos-EXP1}]{\label{fig:FuSinCos-EXP1}\includegraphics[width=0.9\textwidth]{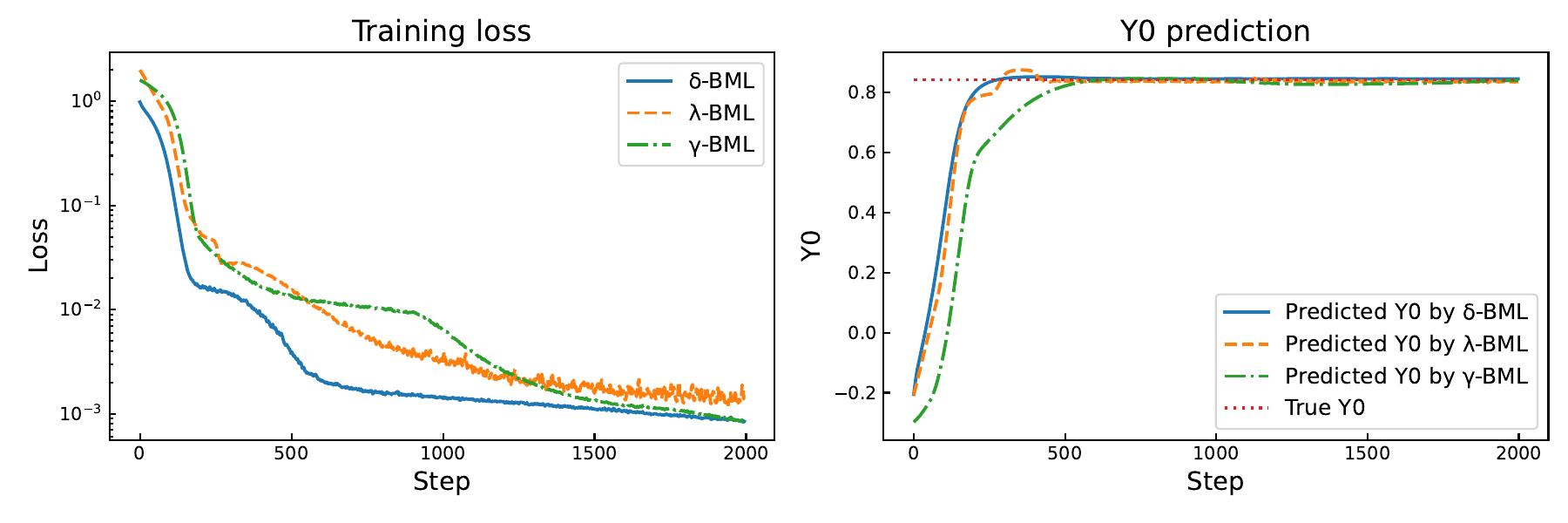}}
    \vskip.05\baselineskip
    \subfigure[Training process of Example~\ref{example:LongSin-EXP1}]{\label{fig:LongSin-EXP1}\includegraphics[width=0.9\textwidth]{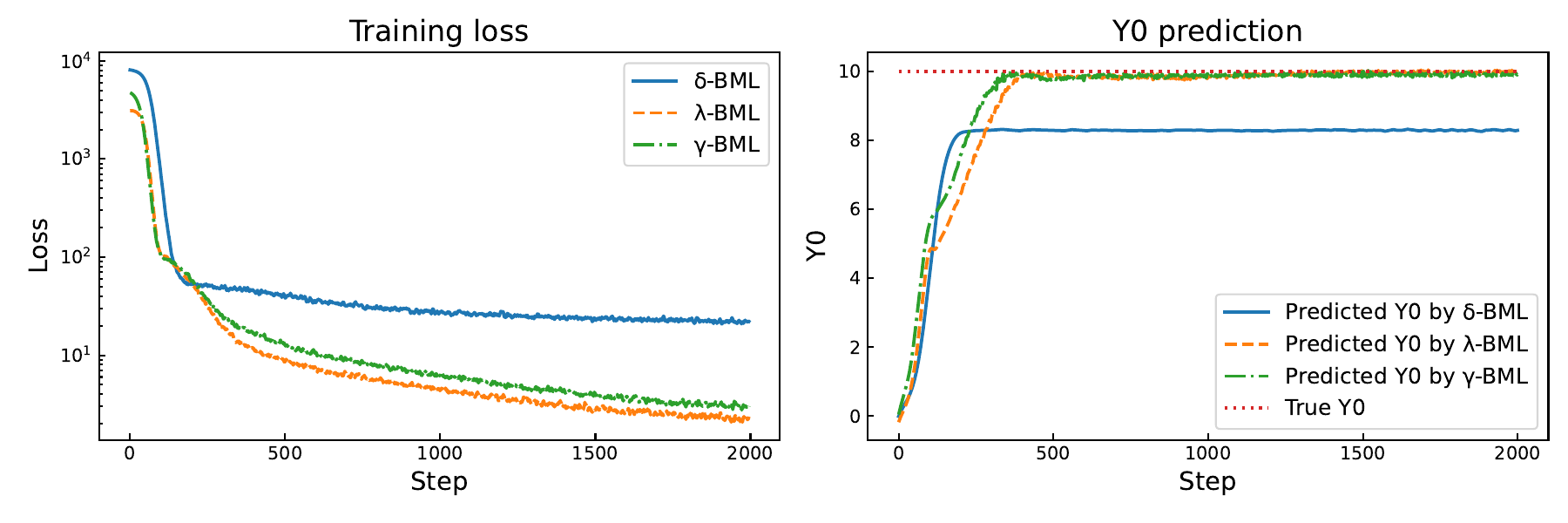}}
    \vskip.05\baselineskip
    \subfigure[Training process of Example~\ref{example:JiLQ-EXP1}]{\label{fig:JiLQ-EXP1}\includegraphics[width=0.9\textwidth]{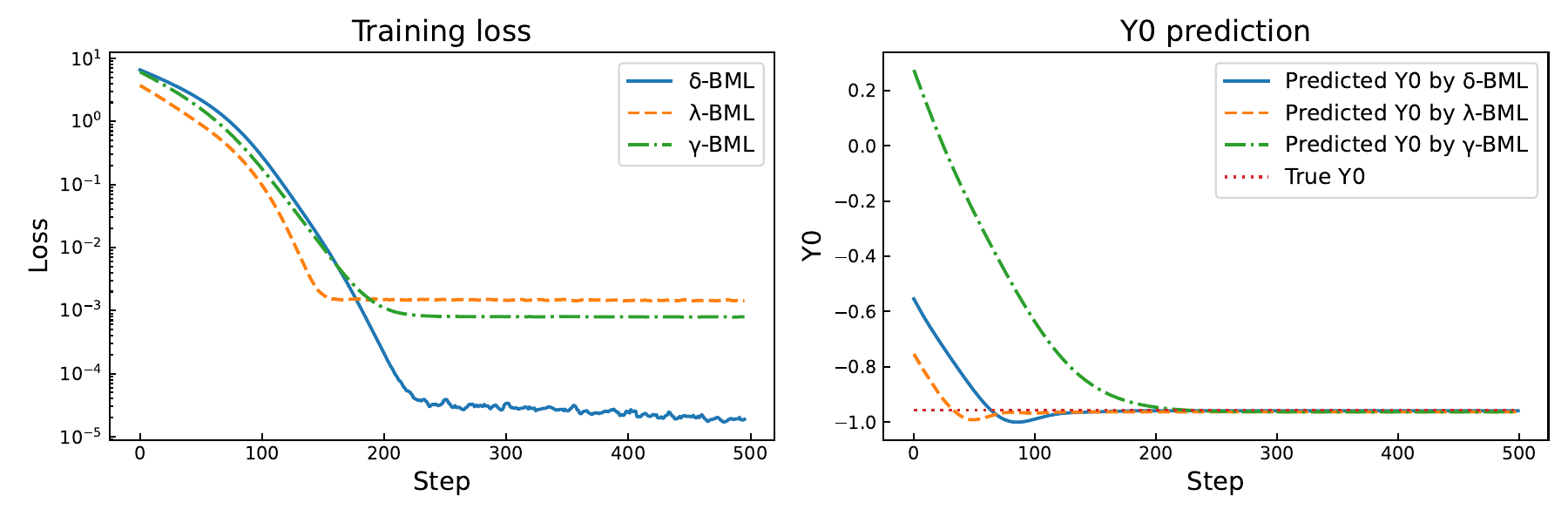}}
    \vskip.05\baselineskip
    \subfigure[Training process of Example~\ref{example:JiLQ-EXP2}]{\label{fig:JiLQ-EXP2}\includegraphics[width=0.9\textwidth]{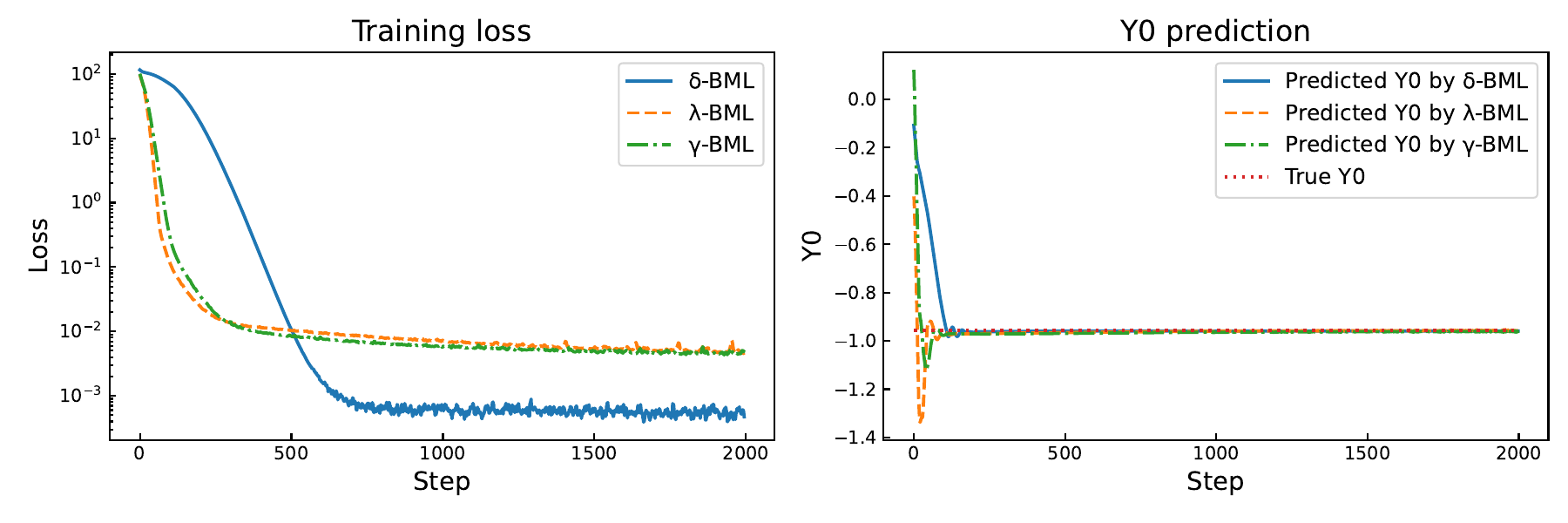}}
    \caption{Training processes for optimizing the three BML values. The $x$-axis represents the number of gradient updates.
    The $y$-axis in the two columns corresponds to the BML value and relative error of $Y_0$, respectively.}
    \label{fig:all-training-plots}
\end{figure}

\begin{table}
\centering
\caption{Results for different examples.}
\label{tab:solveFBSDE-results}
\begin{tabular}{ccrrr}
\toprule
& Objective & Predicted $Y_0$ & Relative Error \\
\hline
Example~\ref{example:FuSinCos-EXP1} & $\delta$-BML & 0.8443 & 0.33\%   \\ 
& $\lambda$-BML & 0.8352 & 0.75\% & \\ 
& $\gamma$-BML & 0.8396 & 0.24\% &  \\ \hline
Example~\ref{example:LongSin-EXP1} & $\delta$-BML & 8.2888 & 17.1\% \\ 
& $\lambda$-BML & 9.9921 & 0.08\% & \\ 
& $\gamma$-BML & 9.9232 & 0.76\% & \\ \hline
Example~\ref{example:JiLQ-EXP1} & $\delta$-BML & -0.9589 & 0.21\%   \\ 
& $\lambda$-BML & -0.9632 & 0.68\%  \\ 
& $\gamma$-BML & -0.9627 & 0.62\%  \\ \hline
Example~\ref{example:JiLQ-EXP2} & $\delta$-BML & -0.9593 & 0.25\%  \\ 
& $\lambda$-BML & -0.9558 & 0.08\%  \\ 
& $\gamma$-BML & -0.9603 & 0.37\%  \\
\bottomrule
\end{tabular}
\end{table}

\begin{example}\label{example:FuSinCos-EXP1}
  Consider the following 1-dimensional fully coupled nonlinear FBSDE
    \begin{equation*}
      \left\{\begin{aligned} X_t= & x_0 - \int_0^t
      \frac{1}{2}\sin(s+X_s)\cos(s+X_s)(Y_s^2+Z_s)\,ds \\ & + \int_0^t
      \frac{1}{2}\cos(s+X_s)(Y_s\sin(s+X_s)+Z_s+1) \,dW_s, \\ Y_t= &
      \sin(T+X_T) +\int_t^T Y_sZ_s - \cos(s+X_s) \,ds - \int_t^TZ_s\,d
      W_s,
      \end{aligned}\right.
    \end{equation*}
    which comes from the literature \cite{fu_zhao_2016,ji_peng_2020_IIS}.

    It is not difficult to verify that the solution to this equation satisfies
    relations
    \begin{equation*}
      Y_t = \sin(t+X_t),\qquad\qquad Z_t = \cos^2(t+X_t).
    \end{equation*}
    Fix the initial condition $x_0=1.0$ and the terminal time $T=1.0$. In this
    case, we have $Y_0\approx 0.841$.
\end{example}

For this example, we choose a discrete time step $h=1.0/25=0.04$; the number of
samples used for calculating the expectation in a single estimation is $M=4096$;
the fully connected neural network has 2 layers with 8 neurons in each layer;
the learning rate is set to 0.001. The optimization processes for the three BML
values are shown in Figure~\ref{fig:FuSinCos-EXP1}.

\begin{example}\label{example:LongSin-EXP1}
  Consider the following forward equation FBSDE
    \begin{equation*}
      \left\{\begin{aligned} X_t= & x_0 + \int_0^t \sigma_0 Y_s\,dW_s,
      \\ Y_t= & \frac{10}{d} \sum_{j=1}^d \sin \left(X_{j, T}\right)
      +\int_t^T-r Y_s+\frac{ \sigma_0^2}{2} e^{-3 r(T-s)}\left(
      \frac{10}{d} \sum_{j=1}^d \sin \left(X_{j,
        s}\right)\right)^3\,ds-\int_t^TZ_s\,d W_s,
      \end{aligned}\right.      
    \end{equation*}
    which is taken from the literature \cite{bender_zhang_2008,han_long_2020}.

    It can be verified by It\^o's formula that the solution to this equation
    satisfies the following relations
    \begin{equation*}
      \left\{
      \begin{aligned}
        Y_t &= \frac{10}{d} e^{-r(T-t)}\sum_{j=1}^d \sin \left(X_{j,
          s}\right),\\ Z_t &= \sigma_0 \left( \frac{10}{d} \right)^2
        e^{-2r(T-t)}\cos X_t\sum_{j=1}^d \sin \left(X_{j, s}\right).
      \end{aligned}
      \right.
    \end{equation*}
    Set the dimension as $n=d=4$ and $m=1$. Fix the initial condition $x_0= (
    \frac{\pi}{2}, \frac{\pi}{2}, \frac{\pi}{2}, \ldots,
    \frac{\pi}{2})^\intercal\in\mathbb{R}^n$ and the terminal time $T=1.0$. The
    parameters in the equation are chosen as $r=0$ and $\sigma_0=0.4$. In this
    case, we have $Y_0=10$.
\end{example}

For this example, we choose a discrete time step $h=1.0/50=0.02$; the number of samples used for calculating the expectation in a single estimation is $M=1024$; the fully connected neural network has 3 layers with 32 neurons in each layer; the learning rate is set to 0.001. The optimization processes for the three BML values are shown in Figure~\ref{fig:LongSin-EXP1}.

\begin{example}\label{example:JiLQ-EXP1}
  Consider the following stochastic linear quadratic problem
  \begin{align*}
    \min_{u_t}&\quad \operatorname{\mathbb{E}}\biggl[
      \frac{1}{2}X_T^\intercal QX_T + \int_0^T \Bigl(\frac{1}{4}\|X_t\|^2 +
      \|u_t\|^2\Bigr)\,dt \biggr],\\ \operatorname{s.t.}&\quad X_t = x_0 +
    \int_o^t(- \frac{1}{4}X_t + u_t)\,dt + \int_0^t (\frac{1}{5}X_t +
    u_t)\,dW_t,
  \end{align*}
    which is taken from the literature \cite{ji_peng_2022}.

  The dimensions are set as $n=m=5, d=1$. The initial condition is fixed at
  $x_0=\mathbf{1}_n$, a vector with all elements equal to 1 in $\mathbb{R}^n$.
  The terminal time is set to $T=0.1$. The weight matrix $Q$ is chosen as the
  identity matrix.

  The maximum principle gives the following first-order adjoint equations for this problem
  \begin{equation*}
  \left\{
  \begin{aligned}
    X_t &= x_0 + \int_0^t (- \frac{1}{4}X_t + \frac{1}{2}Y_t +
    \frac{1}{2}Z_t) \,ds + \int_0^t ( \frac{1}{5}X_t + \frac{1}{2}Y_t
    + \frac{1}{2}Z_t) \,dW_s,\\ Y_t &= -QX_T + \int_t^T(-
    \frac{1}{2}X_t - \frac{1}{4}Y_t + \frac{1}{5}Z_t)\,ds -
    \int_t^TZ_s\,dW_s.
  \end{aligned}
  \right.
  \end{equation*}
  If this system can be solved, the optimal control is $u^*=(Y^*+Z^*)/2$. By
  solving the Riccati equation for the original stochastic optimal control
  problem, it can be deduced that $Y_0\approx-0.9586 x_0$ \cite{ji_peng_2022}.
\end{example}

For this example, we choose a discrete time step $h=0.1/25=0.004$; the number of
samples used for calculating the expectation in a single estimation is $M=64$;
the fully connected neural network has 2 layers with 16 neurons in each layer;
the learning rate is set to 0.001. The optimization processes for the three BML
values are shown in Figure~\ref{fig:JiLQ-EXP1}.

\begin{example}\label{example:JiLQ-EXP2}
Consider Example~\ref{example:JiLQ-EXP1} when $n=m=100$, 
which transforms it into an optimal control problem with a state dimension of 100.
Other settings remain the same.
\end{example}

For this example, we retain the basic settings for dealing with the
corresponding low-dimensional problem, and only adjust the learning rates for
$\delta$-BML, $\lambda$-BML, and $\gamma$-BML to $5\times10^{-4}$,
$2\times10^{-3}$, and $2\times10^{-3}$, respectively. The results are shown in
Figure~\ref{fig:JiLQ-EXP2}.

From Figure~\ref{fig:all-training-plots}, it can be clearly observed that in all
the four examples, the three BML values stabilize after 2000 gradient descents, and
the predicted values $Y_0$ converge to and remain close to the reference values
after 500 gradient descents. The training results reported in
Table~\ref{tab:solveFBSDE-results} reveal that almost all relative errors reach
an accuracy of $10^{-3}$, except for the $\delta$-BML method in
example~\ref{example:LongSin-EXP1}, which fails to converge properly. We believe
that this may be due to the optimization method not successfully reducing the value
of $\delta$-BML to a small enough value, as can be observed from
Figure~\ref{fig:LongSin-EXP1}, where the final value of $\delta$-BML is an order
of magnitude larger than the other two BML values. Further analysis of the
underlying reasons awaits future research.

It is worth noting that even with an increase in state dimension from 5 to 100,
as seen from Example~\ref{example:JiLQ-EXP1} to Example~\ref{example:JiLQ-EXP2},
comparable estimation accuracy can be achieved by fine-tuning the learning rates
and increasing the number of training steps to ensure that the BML values can be
optimized to a sufficiently small value.

\section{Conclusion}

This paper proposes and studies the backward measurability loss (BML) for general coupled FBSDEs. We have analyzed the theoretical properties of the proposed BML and illustrated its physical interpretation through Picard iterations. Exploiting the property that the BML value equals zero if and only if the trial solution coincides with the equation's solution, we have developed an algorithm for numerically solving FBSDEs based on optimizing the BML value. The effectiveness of the proposed algorithm has been demonstrated through numerical experiments on two coupled nonlinear FBSDEs and two stochastic optimal control problems. 

\bibliographystyle{plainnat}
\bibliography{solveFBSDE}

\end{document}